\documentclass[11pt]{ip-journal}

\usepackage{amssymb,amsfonts, amsmath}

\usepackage{mathtools}
\usepackage{color}

\newtheorem{theorem}{Theorem}[section]
\newtheorem{claim}[theorem]{Claim}
\newtheorem{corollary}[theorem]{Corollary}
\newtheorem{definition}[theorem]{Definition}
\newtheorem{example}[theorem]{Example}
\newtheorem{lemma}[theorem]{Lemma}
\newtheorem{proposition}[theorem]{Proposition}
\newtheorem{remark}{Remark}[section]
\newtheorem{conjecture}{Conjecture}[section]

\newcommand{\dd}{\partial \overline{\partial} }

\title [Line arrangements] {Calabi-Yau metrics with conical singularities along line arrangements}
\author[M. de Borbon]{Martin de Borbon}
\author[C. Spotti]{Cristiano Spotti}
\address{King's College London}
\email{martin.deborbon@kcl.ac.uk}
\address{Aarhus University}
\email{c.spotti@qgm.au.dk}

\begin{document}

\begin{abstract}
	Given a finite collection of lines \(L_j \subset \mathbb{CP}^2\) together with real numbers \( 0 < \beta_j < 1 \) satisfying natural constraint conditions, we show the existence of a Ricci-flat K\"ahler metric \(g_{RF}\) with cone angle \(2\pi\beta_j\) along  each line \(L_j\)  asymptotic to a polyhedral K\"ahler cone at each multiple point. Moreover, we discuss a Chern-Weil formula that expresses the energy of \(g_{RF}\) as a logarithmic Euler characteristic with points weighted according to the volume density of the metric.
\end{abstract}

\maketitle

\section{Introduction}

The study of singular K\"ahler-Einstein metrics goes back to Yau's seminal work on the Calabi conjecture \cite[p.389]{Yau}; see also \cite{YaurolePDE}, \cite{YauChengChern} and \cite{yauasterisque}. Extensions to the non-compact complete setting go back to \cite{YauChengFeff} and \cite{TianYauI, TianYauII}. In this paper we analyse Calabi-Yau metrics with conical singularities at a union of complex lines and we provide a detailed analysis of their tangent cones at multiple points.

Consider a weighted line arrangement \( \mathcal{L} = \{ (L_j, (1-\beta_j)) \}_{j=1}^n \) in \( \mathbb{CP}^2\). We assume that \(L_j \neq L_k\) if \(j \neq k \) and that \( 0 < \beta_j < 1 \) for all \(j\). Let \( S = \cup_i \{x_i\} \) be the finite set of points where two or more lines intersect. Elements in \(S\) are referred as the multiple points. Write \( \{ (L_{i_k}, (1-\beta_{i_k})) \}_{k=1}^{d_i} \subset \mathcal{L} \) for the lines going through \(x_i \in S\), where the number of lines \(d_i \geq 2\) is the multiplicity of \(x_i\). We assume that for any multiple point with \(d_i \geq 3\) the Troyanov  conditions are satisfied:
\begin{equation} \label{cond1}
\sum_{k=1}^{d_i} (1-\beta_{i_k}) < 2
\end{equation} 
\begin{equation} \label{cond2}
1 - \beta_{i_k} < \sum_{l \neq k} (1-\beta_{i_l}), \hspace{2mm} \mbox{for any }  1 \leq k \leq d_i .
\end{equation}
We restrict to the Calabi-Yau (CY) regime:
\begin{equation} \label{cond3}
\sum_{j=1}^{n}(1-\beta_j) = 3 .
\end{equation}

In affine coordinates centred at \(x_i\) we have \(d_i\) complex lines of the arrangement going through the origin in \( \mathbb{C}^2\) which correspond to \(d_i\) points in the Riemann sphere. If \(d_i \geq 3\) the Troyanov constraints \eqref{cond1} and \eqref{cond2} guarantee the existence of a unique compatible spherical metric on \( \mathbb{CP}^1 \) with cone angles \( \{2 \pi \beta_{i_k} \}_{k=1}^{d_i} \) at these points, see \cite{luotian} and \cite{troyanov2}. The spherical metric lifts through the Hopf map to a polyhedral  K\"ahler (PK) cone metric \(g_{F_i} \) on \( \mathbb{C}^2 \) with its apex located at \(0\) and cone angles \(2\pi\beta_{i_k}\) along the \(d_i\) complex lines, see \cite{panov} and also \cite{dB}. If \(d_i =2\) we set the PK cone  \(g_{F_i}\) to be given by the product \( \mathbb{C}_{\beta_{i_1}} \times \mathbb{C}_{\beta_{i_2}} \) of two one dimensional cones. We denote by \(r_i\) the intrinsic distance to the apex of \(g_{F_i}\). Our main result is the following:

\begin{theorem} \label{thm1}
	There is a unique (up to scaling) Ricci-flat K\"ahler metric \(g_{RF}\) on \(\mathbb{CP}^2\) with cone angle \(2\pi\beta_j\) along \(L_j\) for \(j=1, \ldots, n \)  \emph{asymptotic} to \(g_{F_i}\) at \(x_i\) for each \(i\). More precisely, there is \(\mu>0\) and complex coordinates \(\Psi_i\) centred at \(x_i\) such that  
	\begin{equation} \label{asymptotics}
		| \Psi_{i}^{*} g_{RF} - g_{F_i} |_{g_{F_i}} = O (r_i^{\mu}) \hspace{2mm} \mbox{as } r_i \to 0 .
	\end{equation}
\end{theorem}

In algebro-geometric terms our angle constraints given by equations \eqref{cond1} and \eqref{cond3} are equivalent to say that \((\mathbb{CP}^2, \sum_{j=1}^{n}(1-\beta_j)L_j)\) is a log CY klt pair. This pair admits a unique up to scale weak Ricci-flat singular K\"ahler metric \(g_{RF}\), see \cite{EGZ}. This weak solution has locally defined continuous K\"ahler potentials which admit a polyhomogeneous expansion (see \cite{JMR}, \cite{guenancia}) at points of \(  L_j^{\times} \), where
\begin{equation*}
	L_j^{\times} = L_j \setminus \cup_i \{x_i\} .
\end{equation*}

 From this point of view our Theorem \ref{thm1} provides, \emph{under the extra condition given by Equation \eqref{cond2}}, a more precise description of the asymptotics of the metric \(g_{RF}\) at the multiple points of the arrangement. The constraint specified by Equation \eqref{cond2} can be interpreted as a restriction to the case when there is no `jumping' of the tangent cone or, equivalently,  as a stability condition for \(x_i \in S \) thought of as a singularity of the klt pair.

\begin{remark}
	A couple of remarks are in order.
	\begin{itemize}
		\item Our main theorem can be seen as an analogue  in the logarithmic setting of the  result of Hein and Sun for conifold Calabi-Yau varieties \cite{heinsun}. However, from one hand, we \emph{do not assume smoothability conditions} but, on the other, we crucially restrict ourself to the \emph{two dimensional} log case.
		
		\item In complex dimension \(\geq 3\) the curvature of a (non-flat) Calabi-Yau cone metric blows-up (taking arbitrary large positive and negative values) close to its vertex and the argument presented in this paper breaks down. On the other hand, one might still consider higher dimensional polyhedral K\"ahler cones, but the singular set of the known examples is usually quite complicated (e.g. reflection arrangements as considered in \cite{CouwenbergHeckmanLooijenga}) and the linear analysis presented here also fails to cover such configurations.
	\end{itemize}
	
\end{remark}

The asymptotics in Equation \eqref{asymptotics} give us the following:

\begin{corollary}
	The metric completion of \( (\mathbb{CP}^2 \setminus \cup_{j=1}^n L_j, g_{RF}) \) is homeomorphic to \(\mathbb{CP}^2\). The tangent cone at \(x_i\) agrees with \(g_{F_i}\), and the volume density at this point is equal to 
	\begin{equation} \label{vol dens mult points}
	\nu_i = \begin{cases}
	\gamma_i^2 \quad \hspace{5.5mm} \mbox{if } d_i \geq 3 \\
	\beta_{i_1} \beta_{i_2} \quad  \mbox{if } d_i =2 \\
	\end{cases}
	\end{equation}
	with \( 0 < \gamma_i < 1 \) given by
	\begin{equation} \label{angle mult points}
	2 \gamma_i = 2 + \sum_{k=1}^{d_i}(\beta_{i_k} -1) .
	\end{equation}     
\end{corollary}

\begin{remark}
	Our proof of Theorem \ref{thm1} does not make any substantial use of the fact that the ambient space is \(\mathbb{CP}^2\) and the conical divisor is made up of lines (see Theorem \ref{gen case}). Thus our results  can be extended with minor changes to the more general situation of a \(2\)-dimensional log CY klt pair in which every singular point of the conical divisor is `stable' in the sense that it has a neighbourhood which is mapped to a neighbourhood of the apex of a PK cone via a biholomorphic map that matches the conical sets according to their angles. In the setting of Theorem \ref{thm1}, this condition is guaranteed by Equation \eqref{cond2}.  More generally, our technique extends to cover also, for example, the case of a \emph{cusp singularity} with cone angle \(2\pi\beta\) in the range \(1/6 < \beta< 5/6\). See last section \ref{conjectural section} for more details.
\end{remark}

Let us briefly describe the ideas in the proof of our main Theorem \ref{thm1}.

\subsection*{Outline of proof of Theorem \ref{thm1}}

In  affine coordinates \((\xi_1, \xi_2)\), say,  we have defining equations for the lines  \(L_j  = \{\ell_j =0\} \) given by non-zero affine functions \(\ell_j(\xi_1, \xi_2)\) for \(j=1, \ldots, n\). We have the locally defined holomorphic volume form
\begin{equation} \label{hol vol form def}
\Omega = \ell_1^{\beta_1-1} \ldots \ell_n^{\beta_n-1} d\xi_1 d\xi_2 .
\end{equation}
The real volume form \(dV = \Omega \wedge \overline{\Omega}\) is globally defined and is independent of the choices made up to scaling by a constant positive factor. The requirement that \( \sum_{j=1}^{n} (1-\beta_j) = 3 \) implies that \(dV\) is singular only along the lines \(L_j\) and has poles of the required type. We are looking for a K\"ahler metric\footnote{By standard abuse of language, we identify the K\"ahler metric with its corresponding K\"ahler form: if \(g_{*}\) is the metric then \(\omega_{*} \) is the corresponding \((1,1)\)-form.}  \(\omega_{RF}\) with 
\begin{equation} \label{Ricci flat eq}
	\omega_{RF}^2 = \Omega \wedge \overline{\Omega} .
\end{equation}

First we construct a reference metric \(\omega_{ref}\) which has conical singularities in a H\"older continuous sense as defined by Donaldson \cite{donaldson1} of cone angle \(2\pi\beta_j\) at points of \(L_j^{\times} \) and is isometric to a PK cone (in particular is flat) in a neighbourhood of each multiple point. Around each \(x_i\) we have affine coordinates with \(\Phi_i^{*}g_{ref} = g_{F_i} \). The bisectional curvature of \(g_{ref}\) is uniformly bounded above.

To solve Equation \eqref{Ricci flat eq} we use Yau's continuity path \cite{Yau}, mixing techniques present in different extensions of Yau's work on the Calabi conjecture to the context of: metrics with cone singularities along a smooth divisor (\cite{brendle} and \cite{JMR}), conifolds (\cite{heinsun}) and ALE spaces (\cite{Joyce}).
 
Thus we first develop the relevant linear theory, that is a Fredholm set-up  for the Laplace operator of \(\omega_{ref}\) in weighted H\"older spaces adapted to our singularities. This  parallels the corresponding theory for Riemannian conifolds as developed in \cite{JoyceSL}, \cite{Behrndt}, \cite{heinsun}, \cite{pacini} with the multiple points playing the r\^ole of the conifold points. The arguments in this section rely  in a fundamental way on Donaldson's interior Schauder estimates \cite{donaldson1}. 
The main result in this section is Proposition \ref{linearthm}.

We set up the continuity path and prove the relevant a priori estimates. We let \(T\) be the set of \(t \in [0, 1] \) such that there is \( u_t \) with  \( \omega_t^2 = c_t e^{t f_0} \omega_0^2 \) where \( \omega_t =  \omega_0 + i \dd u_t \), \(\omega_0\) a suitable initial metric and \(\omega_t\) satisfies the asymptotics \eqref{asymptotics} with respect to \(\Phi_{i, t} = \Phi_{i} \circ P_{i}(t)^{-1} \) with \(P_{i}(t) (z_1, z_2) = (\lambda_{i, 1}(t) z_1 , \lambda_{i, 2}(t) z_2 ) \) for some \(\lambda_{i, 1}(t), \lambda_{i, 2}(t) >0 \) and \(\lambda_{i, 1} = \lambda_{i, 2}\) if \(d_i \geq 3\). The theorem follows once we show that the functions \(u_t\) remain uniformly bounded (in a suitable sense) with respect to \(t\). Our initial metric \(\omega_0\) has bounded Ricci curvature and along the continuity path we have \(\mbox{Ric}(\omega_t) = (1-t) \mbox{Ric}(\omega_0) \), therefore we have a uniform lower bound on the Ricci curvature \(\mbox{Ric}(\omega_t) \geq -C \omega_{ref} \). Since \( \mbox{Bisec}(\omega_{ref}) \leq C \) we can appeal to the Chern-Lu inequality  to conclude  that  \( C^{-1} \omega_{ref} \leq \omega_t \leq C \omega_{ref} \), which is the geometric content of the \(C^2\)-estimate and a main point in the proof.
\qed

\subsection*{Applications and perspectives} 

We discuss how the polynomial decay of the CY metric suggests that their energy (i.e., the \(L^2\)-norm of the curvature) can be a-priori computed by considering a log version of the Euler characteristic which takes into account the value of the metric density at the multiple points, see formula \eqref{energy formula}. We compute the logarithmic Euler characteristic in the generic normal crossing case, and moreover we verify that it agrees with previous formulas available in the literature \cite{Tian, panov}.

Finally, we describe the more general picture we expect to hold for K\"ahler-Einstein metrics with conical singularities (KEcs) on surfaces. In particular, we discuss `jumping' tangent cones at singular points of curves and fit our discussion within the theory of normalized volumes of valuations \cite{Li15}. We compute the expected jumping in some typical situations: irreducible case, many lines, non-transverse reducible case, showing how the metric tangent cones should jump when we vary the cone angles parameters.
 
We also propose a differential-geometric derivation of a logarithmic version of the Bogomolov-Miyaoka-Yau inequality, i.e., a generalization of the previously discussed energy formula where we again weight singular points of the curve with their metric density or, equivalently, with the infimum of the normalized volume of valuations, even when jumps are expected to occur, see formula \eqref{CHERNWEILintro}. Somehow remarkably, such formula matches with previous versions obtained by purely algebraic means (see \cite{Langer}), and it would have the advantage to characterize geometrically the equality case with conical metrics of constant holomorphic sectional curvature.

\subsection*{Organization of the paper}
\begin{itemize}
	\item \emph{Background: section \ref{Backround} and appendix \ref{pm sphere section}}. In section \ref{Backround} we recollect preliminary material on K\"ahler metrics with conical singularities along a smooth divisor (following \cite{donaldson1}) and PK cones (following \cite{panov}). The statement of Theorem \ref{thm1} should be clear after this section.  We have also added in appendix \ref{pm sphere section} a short overview on the theory of polyhedral metrics on the projective line, as a sort of precedent to Theorem \ref{thm1}.
	
	\item  \emph{Main result: sections \ref{reference metrics}, \ref{linear thry sct}, \ref{yau cont path sect}}. The proof of Theorem \ref{thm1} take us sections \ref{reference metrics} (reference metric), \ref{linear thry sct} (linear theory) and \ref{yau cont path sect} (continuity method). 
	
	\item \emph{Conjectural picture: sections \ref{chernweilsect} and \ref{conjectural section}}. We discuss the energy formula in section \ref{chernweilsect}, and describe the general picture of KEcs metrics on surfaces, jumping of tangent cones, and applications to a generalized log Bogomolov-Miyaoka-Yau formula in the final section  \ref{conjectural section}.
	
\end{itemize}

\subsection*{Acknowledgments} Both authors are supported by AUFF Starting Grant 24285 and Danish National Research Foundation Grant DNRF95 `Centre for Quantum Geometry of Moduli Spaces'.

\section{Background} \label{Backround}

\subsection{K\"ahler metrics with conical singularities along a smooth divisor} \label{kmcs}

Fix $ 0 < \beta <1$. On $\mathbb{R}^2 \setminus \lbrace 0 \rbrace$ with polar coordinates $(\rho, \theta)$ we have the metric of a cone of total angle $2\pi \beta$ with apex located at $0$
\begin{equation*}
g_{\beta} = d\rho^2 +  \beta^2 \rho^2 d\theta^2 .
\end{equation*}
There is an induced a complex structure on the punctured plane given by an anti-clockwise rotation of angle $\pi/2$ with respect to \(g_{\beta}\). A basic fact is that we can change coordinates so that this complex structure extends smoothly to the origin. Indeed, if
\begin{equation}
z = \rho^{1/\beta} e^{i\theta} \label{CHCO}
\end{equation}
we get $g_{\beta} = \beta^2 |z|^{2\beta -2} |dz|^2$.  We denote by $\mathbb{C}_{\beta}$  the complex numbers endowed with this singular metric. 
On \(\mathbb{C}^2\) we have the product metric \(g_{(\beta)} =  \beta^2 |z_1|^{2\beta -2} |dz_1|^2  +  |dz_2|^2 \)
with cone angle \(2\pi\beta\) along $ \lbrace z_1 =0 \rbrace$. The vector fields
\(v_1 = |z_1|^{1-\beta} \partial_{z_1}\) and \(v_2 = \partial_{z_2}\)
are orthogonal and have constant length with respect to $g_{(\beta)}$.

Let \(X\) be a closed complex surface, \(C \subset X \) a smooth complex curve, $g$ a smooth K\"ahler metric on $X \setminus C$ and $p \in C$. Take complex coordinates  $(z_1, z_2)$ centred at $p$ and adapted to the curve in the sense that $C = \lbrace z_1 =0 \rbrace$. In the complement of $C$ we have smooth functions $G_{i\overline{j}}$ given by $G_{i\overline{j}} = g (v_i, \overline{v}_j)$ for \(1\leq i, j \leq 2 \). 
\begin{definition} \label{def kmcs}
	We say that $g$ has cone angle $2\pi\beta$ along $C$
	if for every \(p\) and adapted complex coordinates \((z_1, z_2)\) as above, the functions  $G_{i\overline{j}}$ admit a  H\"older continuous extension to $C$. We also require the matrix $(G_{i\overline{j}} (p))$ to be positive definite  and that  $G_{1\overline{2}} =0$  when $z_1=0$. In particular, it follows that in adapted complex coordinates \( A^{-1} g_{(\beta)} \leq g \leq A g_{(\beta)} \) for some \(A>0\).
\end{definition}

The fact that Definition \ref{def kmcs} is independent of the choice of adapted complex coordinates can be checked by a straightforward computation, writing \(\tilde{z}_1= f z_1\) with \(f\) a non-vanishing holomorphic function, see page 3 in \cite{dBthesis}.
The metric \(g\) has an associated K\"ahler form \(\omega\) that defines a closed current and therefore a de Rham cohomology class \([\omega]\). Around points in \(C\) there are local K\"ahler potentials \(\omega=i \dd \phi \) with $\phi \in C^{2, \alpha, \beta}$ (see \cite{donaldson1}). The metric \(g\) induces (in the usual way) a distance \(d\) on \(X\setminus C\) and it is straightforward to show that the  metric completion of \((X\setminus C, d)\) is homeomorphic to \(X\). The metric tangent cone at points on the curve is \(g_{(\beta)}\). Furthermore, Definition \ref{def kmcs} is well suited to the development of a Fredholm theory for the Laplace operator of \(g\) in H\"older spaces, see \cite{donaldson1} and section \ref{linear thry sct}. 

There are two types of coordinates relevant in our context. The first is given by adapted complex coordinates $z_1, z_2$. In the second one we replace the coordinate $z_1$ with $\rho e^{i\theta}$ by means of Equation \eqref{CHCO} and extend over the origin by means of the usual change between polar to Cartesian coordinates. We refer to the later as cone coordinates\footnote{In cone coordinates the Laplace operator of \(g_{(\beta)}\) is uniformly elliptic with bounded coefficients which are discontinuous at the origin, see paragraphs after Theorem 3.2 in \cite{haoyin}.}. In other words,  there are two relevant differential structures on $X$ in our situation. One is given by the complex manifold structure we started with, the other  is given by declaring the cone coordinates to be smooth. The two structures are clearly equivalent by a map modelled on \( (\rho e^{i\theta}, z_2) \to (\rho^{1/\beta} e^{i\theta}, z_2)  \) in a neighbourhood of $C$. Note that the notion of a function being H\"older continuous (without specifying the exponent) is independent of the coordinates we take, therefore there is no ambiguity in Definition \ref{def kmcs}.

\subsubsection*{Bisectional curvature}
Let \(g\) be a K\"ahler metric given in local complex coordinates by \(g_{i\bar{j}} = g (\partial/\partial z_i, \partial/ \partial \bar{z}_j)\). Its Riemannian curvature tensor is given by
\begin{equation}\label{eq:curvature}
R_{i\overline{j}k\overline{l}} = - g_{i\overline{j}, k\overline{l}} + g^{p\overline{q}} g_{i\overline{q}, k} g_{p\overline{j}, \overline{l}} 
\end{equation}
where \(R_{i\bar{j}k\bar{l}} = R(\partial_i, \partial_{\bar{j}}, \partial_k, \partial_{\bar{l}})\). 
Let \(u\) and \(v\) be complex vector of type \((1,0)\), following \cite{tianbook} we define the bisectional curvature as 
\begin{equation}\label{eq:bisec}
	\mbox{Bisec}(u, v) = R(u, \bar{u}, v, \bar{v}).
\end{equation}
In particular, if \(u=v\) then \(\mbox{Bisec}(u,u)\) measures the sectional curvature of the complex line spanned by the vector \(u\). The next result is classical and follows from the Gauss-Codazzi equations applied to complex submanifolds of a K\"ahler manifold, see Section 4 in \cite{GoldbergKobayashi}, for completeness we include a proof of it.

\begin{lemma}\label{lem:biseccurv}
	Let \(g\) be a K\"ahler metric on a complex manifold \(X\) and let \(D \subset X\) be a smooth complex hypersurface endowed with its induced K\"ahler metric \(g|_D\). Then the bisectional curvature of \(g|_D\) is less or equal than the bisectional curvature of \(g\). In other words, if \(u, v \in T^{1,0}_p D\) then \(\mbox{Bisec}_{D}(u, v) \leq \mbox{Bisec}_X (u, v)\).
\end{lemma}

\begin{proof}
	Let \((z_1, \ldots, z_n)\) be complex coordinates centred at \(p\) such that \(D=\{z_1=0\}\). By making a change in the \(z_2, \ldots, z_n\) variables we can assume that \((z_2|_D, \ldots,z_n|_D)\) are normal coordinates centred at \(p\) for \(g|_D\), in other words the derivatives \(\partial g_{i\bar{j}} / \partial z_k\) and \(\partial g_{i\bar{j}} / \partial \bar{z}_k\)	vanish at the origin when \(i,j,k \geq 2\).
	Moreover, we can assume that the coordinate vector fields form a unitary frame at the origin and that \(u=\partial/ \partial z_i (0)\) and \(v = \partial / \partial z_j (0)\) with \(i, j \geq 2\). It follows from Equations \eqref{eq:curvature} and \eqref{eq:bisec} that
	\[\mbox{Bisec}_X (u, v) = \mbox{Bisec}_D (u, v) + |g_{i\bar{1}, j}(0)|^2 . \qedhere \]
\end{proof}

\subsubsection*{Reference metrics}
The following lemma provides us with a large class of examples of metrics with conical singularities:

\begin{lemma} \label{CSing}
	Let \(\eta\) be a smooth \((1, 1)\)-form and \(F\) a smooth positive function on \(B\) (the unit ball in \(\mathbb{C}^2\)), assume that \( \eta (\partial_{z_2}, \overline{\partial_{z_2}}) (0) >0 \) and define
	\begin{equation}
	\omega = \eta + i \dd (F|z_1|^{2\beta}) .
	\end{equation}
	There is \(\epsilon > 0\) such that \(\omega\) is a K\"ahler metric on \(\epsilon B\) with cone angle \(2\pi\beta\) along \(\{z_1=0\}\) in the sense of Definition \ref{def kmcs}. Moreover, if \(\eta\) is positive on \(B\) then  the  bisectional curvature of \(\omega\) is uniformly bounded above on \((\epsilon/2) B\).
\end{lemma}

\begin{proof}
	The statement about the conical singularities follows from the computation
	$$ i \partial \overline{\partial} (F |z_1|^{2\beta}) = |z_1|^{2\beta} i \partial \overline{\partial} F $$
	$$+ \beta |z_1|^{2\beta-2} \left( \overline{z}_1 idz_1 \overline{\partial}F + z_1 i\partial F d\overline{z}_1 \right) + \beta^2 F |z_1|^{2\beta -2} idz_1 d\overline{z}_1 . $$
	
	If \(\eta>0\) then $ \Gamma = \eta + i\partial \overline{\partial} (F |z_{3}|^2)  $
	defines a smooth K\"ahler metric in a neighbourhood of $0$ in $\mathbb{C}^{3}$. Remove a ray in the complex plane corresponding to the $z_1$ variable and let $ \Phi (z_1, z_2) = (z_1, z_2, z_1^{\beta}). $
	The pullback of $\Gamma$ by $\Phi$ is independent of the branch of $z_1^{\beta}$ that we use and $\omega = \Phi^{*} \Gamma$. The upper bound on the bisectional curvature  follows from Lemma \ref{lem:biseccurv}. This argument is presented in \cite{rubinstein}.
\end{proof}
Lemma \ref{CSing} give us the following class of `standard reference metrics':

\begin{example} \label{REFERENCE}
	Let $\eta$ be a smooth K\"ahler form on the closed complex surface $X$ and let $h$ a smooth Hermitian metric on the line bundle $\mathcal{O}(C)$ where \(C\subset X\) is a smooth complex curve. Let $ s \in H^0(\mathcal{O}(C))$ be such that $C = \{ s=0 \} $. For $\delta>0$ set $\omega = \eta + \delta i \partial\overline{\partial} |s|^{2\beta}_h$. If we take $\delta$ small enough then $\omega$ defines a K\"ahler metric with cone singularities as in Definition \ref{def kmcs} in the same de Rham cohomology class as \(\eta\). The metric $\omega$ is, up to quasi-isometry, independent of the choices of $\eta$, $\delta$, $s$ and $h$. The  bisectional curvature of $\omega$ is uniformly bounded from above.
\end{example}

The upper bound on the bisectional curvature of the reference metric \(\omega\) plays a crucial role in the proof of existence of K\"ahler-Einstein metrics (see \cite{JMR}). 
It follows from the computations in the appendix of \cite{JMR} that the  sectional curvature of \(\omega\) is unbounded below when \(\beta>1/2\); since we have an upper bound on the sectional curvature  it then follows that the Ricci curvature of \(\omega\) must be unbounded below. In general, it is conjectured that the if a K\"ahler-Einstein metric (or more generally a conic metric admitting a polyhomogeneous expansions)  with cone angle \(2\pi\beta\) for some \(\beta>1/2\) has bounded curvature then there is a holomorphic splitting \(TX|_D=TD \oplus N_D\) where \(N_D\) denotes the normal bundle; see \cite{Arezzo} for related work.

\subsection{Polyhedral K\"ahler cones (\cite{panov})}  \label{PK section}

Let \(L_1, \ldots, L_d\) be \(d \geq 2\) distinct complex lines in \( \mathbb{C}^2 \) going through the origin and \( 0 < \beta_j < 1 \) satisfy the Troyanov conditions  \eqref{cond1} and \eqref{cond2}, which we rewrite concisely as: 
\begin{equation} \label{Tcondition}
0 < 2 + \sum_{j=1}^{d} (\beta_j -1) < 2 \min_{j} \beta_j
\end{equation}
if \( d \geq 3 \) and \(\beta_1 = \beta_2 \) if \(d=2\). \footnote{In the context of Theorem \ref{thm1} we use PK cones of the form \(\mathbb{C}_{\beta_1} \times \mathbb{C}_{\beta_2}\) at double points of the arrangement. We \emph{don't require} that \(\beta_1=\beta_2\) at double points; i.e. there might be no metric on \(\mathbb{C P}^1\) with two cone points of the given angles.} 
It follows from \cite{troyanov2} and \cite{luotian} that there is a unique compatible metric \(g\) on \(\mathbb{CP}^1\) with constant Gaussian curvature  \(4\) and cone angle \(2\pi\beta_j\) at the point corresponding to \(L_j\) for \(j= 1, \ldots, d \) (uniqueness is up to pull-back by dilations if \(d=2\)). 

We recall the construction from Section 3.3 in \cite{panov}.
We  cut $g$ along geodesic segments with vertices at all the conical points to obtain a contractible polygon $P$ which is isometrically immersed by its enveloping map in \(S^2(1/2)\), the round \(2\)-sphere of radius \(1/2\). It follows from Gauss-Bonnet that the total area of \(P\) is equal to \(\pi \gamma\) where \( 0 < \gamma < 1 \) is given by
\begin{equation} \label{number c}
2 \gamma = 2 + \sum_{j=1}^{d} (\beta_j -1) .
\end{equation}

Let \( H: S^3 \to \mathbb{CP}^1 \) be the Hopf \(S^1\)-bundle.
Recall that \(H\) is a Riemannian submersion from \(S^3(1)\) to the projective line with its Fubini-Study metric \(\mathbb{CP}^1 \cong S^2(1/2)\).  Since \( P \subset S^2(1/2) \) is  contractible,  the universal cover of $ H^{-1}(P) \subset S^3(1)$ is diffeomorphic to \( P \times \mathbb{R} \) and its inherited constant curvature $1$ metric is invariant under translations in the \( \mathbb{R} \) factor. The planes orthogonal to the fibres define a horizontal distribution, hence a connection $\nabla$ on \( P \times \mathbb{R} \). The holonomy of $\nabla$ along \(\partial P\)  is equal to the parallel translation by twice the  area of \(P\). On the other hand for any $l >0$ we can take the quotient of \( P \times \mathbb{R} \) by \( l \mathbb{Z} \) to obtain a metric $\overline{g}$ of constant curvature $1$ on \( P \times S^1 \) such that all the fibres are geodesics of length $l$. Consider the metric $\overline{g}$ on \( P \times S^1 \) with $l = 2 \mbox{Area}(P)$. With this choice, the holonomy of the fibration along $\partial P$ is trivial and it makes one full rotation. We let $s$ be a parallel section 
of \(P \times S^1\) and identify points \((x, s(x)) \sim (y, s(y))\) whenever \(x, y \in \partial P\) are identified under \(g\). The quotient of \(P \times S^1\) by this identification is homeomorphic to the \(3\)-sphere because the section \(s\) gives one full rotation along \(\partial P\) and we obtain a metric $\overline{g}$ on \(S^3\). The upshot is that \(\overline{g}\) is a constant sectional curvature \(1\) metric on \(S^3\)  with cone angles \(2\pi\beta_j\) along the Hopf circles \(L_j\) and \(H: (S^3, \overline{g}) \to (\mathbb{CP}^1, g)\) is a Riemannian submersion with fibres of  constant length \(2\pi \gamma\) (so the total volume of \(\overline{g}\) is \(2\pi^2\gamma^2\)). 

We write \(Y = S^3\) and consider the Riemannian cone with link \((Y, \overline{g})\), that is the space \(C(Y) = \mathbb{R}_{>0} \times Y \) endowed with the metric \(g_F = dr^2 + r^2 \overline{g} \). There is a parallel complex structure and canonical global complex coordinates \((z, w)\) which give a biholomorphism \(\overline{C(Y)} \cong \mathbb{C}^2 \) and map the singular locus of \(g_F\) to the complex lines \(L_1, \ldots, L_d\) we started with. The radius function is homogeneous in complex coordinates. For every \(\lambda>0\)
\begin{equation}
	r(\lambda z, \lambda w) = \lambda^{\gamma} r(z, w) ,  \hspace{2mm} r \partial_r = \frac{1}{\gamma} \mbox{Re} \left( z \partial_z + w \partial_w \right) .
\end{equation}
We denote the associated K\"ahler form by \(\omega_F = (i/2) \dd r^2 \) and the Reeb vector field by \(\xi= J (r \partial_r) = (1/\gamma) \mbox{Im} \left( z \partial_z + w \partial_w \right) \).

Following \cite{panov} a PK cone \(C(Y)\) is a Polyhedral K\"ahler manifold which is also a metric cone. These are divided into three types according to the behaviour of the orbits of the Reeb vector field \(\xi\):  
\begin{enumerate}
	\item Regular. These agree with the \(g_F\) described above. 
	
	\item Quasi-regular. These are pull-backs of a regular cone by some branched degree \(pq\) map \( (z, w) \to (z^p, w^q) \) with \(p\) and \(q\) positive integers.
	
	\item Irregular. \(C(Y) = \mathbb{C}_{\beta_1} \times \mathbb{C}_{\beta_2} \) with \(\beta_1 / \beta_2 \) irrational. We also allow \(\beta_2 = 1 \).
\end{enumerate}

\begin{remark}
	The division into regular/quasi-regular/irregular  distinguishes the cases when the Reeb vector field \(Jr\partial_r\) generates a free \(S^1\)-action, a locally free \(S^1\)-action or when it has a non-closed orbit respectively. This follows standard terminology in Sasakian geometry, see \cite{Sparks}.
\end{remark}

The outcome is that PK cones correspond to spherical metrics with cone singularities in the projective line by means of the K\"ahler quotient construction with respect to the $S^1$-action generated by the Reeb vector field, except in the irregular case. The metric on \(\mathbb{CP}^1\) is lifted through the Hopf and Seifert maps in the regular and quasi-regular cases respectively.
The \emph{volume density} of a PK cone with link \((Y, \overline{g})\) is defined as
\begin{equation}
	 \nu = \frac{\mbox{Vol}(\overline{g})}{\mbox{Vol}(S^3(1))},
\end{equation} 
where \(\mbox{Vol}(S^3(1))= 2\pi^2 \) is the total volume of the round \(3\)-sphere of radius \(1\). In the regular case \( \nu = \gamma^2 \) with \(\gamma\) given by Equation \eqref{number c}. For quasi-regular cones \( \nu = pq \tilde{\nu} \)  with \( \tilde{\nu} \) the volume density of the corresponding regular cone. In the irregular case \( \nu = \beta_1 \beta_2 \).

\section{Reference metrics} \label{reference metrics}

\subsection{Construction of the reference metric}\label{sect:refmet}
Let \(h\) be a smooth Hermitian metric on \( \mathcal{O}(1) \) and \( s_j \) be  holomorphic sections of \( \mathcal{O}(1) \) with \( s_j^{-1}(0) = L_j \). Let \(\eta\) be a smooth K\"ahler metric on \(\mathbb{CP}^2\). 
\begin{claim} \label{claim ref 1}
	There is a constant \(C>0\) such that
\begin{equation}
	i\dd (|s_1|_h^{2\beta_1} \ldots |s_n|_h^{2\beta_n}) > - C \eta
\end{equation}	
\end{claim}

\begin{proof}
If \(u\) is a real positive function we have the identity
	\[ i \dd u = u i \dd \log u + u^{-1} i \partial  u \wedge \overline{\partial} u \geq u i \dd \log u . \]
We let \( u = |s_1|_{h}^{2\beta_1} \ldots |s_d|_h^{2\beta_d}  \) and note that each \(i \dd \log |s_j|_h^2\) is a smooth form on the complement of \(L_j\) which admits a smooth extension to \(\mathbb{CP}^2\) (and therefore is bounded) as the standard representative of \(-2\pi c_1 (\mathcal{O}(1)) \) determined by the curvature of \(h\).
\end{proof}
Let \(\Phi_i\) be affine coordinates centred at the multiple points. Let \(B_i\) be balls centred at \(x_i\) such that \(3B_i\) are pairwise disjoint.
Let \(\chi\) be a smooth function on \(\mathbb{CP}^2\) with values in the interval \([0, 1]\) with \(\chi \equiv 0 \) on \( \overline{\cup_i B_i} \), \(\chi>0\) on \( \mathbb{CP}^2 \setminus \overline{\cup_i B_i} \)  and \(\chi \equiv 1 \) on \(\mathbb{CP}^2 \setminus \cup_i 2B_i\). Let \(r^2\) be the K\"ahler potential for the corresponding \(PK\) cone metric at \(x_i\) which we think as defined on \(3B_i\) via \(\Phi_i\).
\begin{claim} There is \(C>0\) such that \label{claim ref 2}
	\begin{equation} \label{ref met con set}
		 i \dd ( (1-\chi) r^2 + \chi |s_1|^{2\beta_1} \ldots |s_n|_h^{2\beta_n} ) > - C \eta
	\end{equation}
\end{claim}

\begin{proof}
It follows from claim \ref{claim ref 1} that we only need to verify 	the assertion on the set \(\cup_i (2B_i \setminus B_i) \). On the other hand at each point \(p\) of \(L_j\) lying on the compact annulus \( 2\overline{B}_i \setminus B_i  \) we can write (in complex coordinates such that \(L_j = \{z_1 = 0\} \)) \( u = F |z_1|^{2\beta_j} \) where \( u = (1-\chi) r^2 + \chi |s_1|^{2\beta_1} \ldots |s_n|_h^{2\beta_n} \) and \(F\) is a smooth positive function uniformly bounded below in a neighbourhood of \(p\). The claim follows from \( i \dd u \geq u i \dd \log u = u i \dd F  \).
\end{proof}

Similarly to \cite{AP}, we have:
\begin{claim}
	There is a smooth non-negative \((1, 1)\)-form \(\eta_0 \geq 0\) such that \( \eta_0 = \eta \) on \( \mathbb{CP}^2 \setminus \cup_i 2^{-1}B_i \) and \(\eta_0 \equiv 0 \) on \(\cup_i \epsilon B_i \) for some $\epsilon >0$.
\end{claim}

\begin{proof}
	We can assume that in the coordinates \(\Phi_i\) we have \( \eta = i \dd \phi \) with \(\phi = |z|^2 + \tilde{\phi} \) with \( \tilde{\phi} = O (|z|^3) \). Let \(\psi\) be a standard cut-off function with \(\psi (t) =0 \) for \(t \leq 1\) and \(\psi (t) =1 \) for \(t \geq 2\). For \(\delta>0\) define \(\psi_{\delta}(t) = \psi (\delta^{-1}t) \) and set \(\phi_{\delta} = |z|^2 + \psi_{\delta} (|z|) \tilde{\phi} \). Note that \(|\dd (\psi_{\delta} \tilde{\phi})| = O(\delta) \). If we take \(\delta\) sufficiently small then (in an obvious notation) \(\eta_{eu} = i \dd \phi_{\delta} \) is positive, it agrees with the euclidean metric on \( \cup_i \delta B\) and \(\eta_{eu} = \eta\) on \(\mathbb{CP}^2 \setminus \cup_i 2\delta B_i \). 
	
	Let \( \epsilon >0 \) and set \(\Psi\) a smooth positive convex function with \(\Psi(t)\) constant for \(t\leq \epsilon^2 \) and \(\Psi (t) = t\) for \(t \geq 4\epsilon^2\). Let \( t = |z|^2 \), so that \(i \dd \Psi (|z|^2) = (\Psi''(t)t + \Psi'(t) ) i \dd |z|^2 \geq 0 \) agrees with the euclidean metric on the complement of the ball of radius \(2\epsilon\) and vanishes on the ball of radius \(\epsilon\). Finally we take \(\epsilon << \delta\) and set (in an obvious notation) \(\eta_0 = i \dd (\Phi \circ \phi_{\delta})  \)
\end{proof}

We are now in position to define our reference metric, for \(\delta>0\) we set
\begin{equation} \label{ref met definition}
	\omega_{ref} = \eta_0 + \delta i \dd ( (1-\chi) r^2 + \chi |s_1|^{2\beta_1} \ldots |s_n|_h^{2\beta_n} )
\end{equation}

\begin{lemma} \label{ref met lemma}
	If \(\delta>0\) is sufficiently small then \(\omega_{ref}\) is a K\"ahler metric with cone angle \(2\pi\beta_j\) along \(L_j\) and \( \Phi_i^{*} \omega_{ref} = \omega_F^{(i)} \) in a neighbourhood of each \(x_i\). Moreover \(\mbox{Bisec}(g_{ref}) \leq C \).
\end{lemma}

\begin{proof}
	If \(\delta\) is sufficiently small then our claim \ref{claim ref 2} implies that \(\omega_{ref}>0\). Since \(\eta_0 \equiv 0\) in small balls around the multiple points, \(\omega_{ref}\) agrees with the corresponding PK cone metric in a neighbourhood of each \(x_i\). The metric \(\omega_{ref}\) has cone angle \(2\pi\beta_j\) along \(L_j \setminus \{x_i\}\) in the sense of Definition \ref{def kmcs}, as follows from Lemma \ref{CSing} in section \ref{kmcs}.
	
	The metric \(\omega_{ref}\) is  flat on \(\cup_i \epsilon B_i\). Since \(\eta_0 > 0\) on the complement of \(\cup_i \epsilon B_i \), Lemma \ref{CSing} gives us a uniform upper bound on the bisectional curvature of \(\omega_{ref}\) on any compact subset of \( \mathbb{CP}^2 \setminus \cup_i \epsilon \overline{B} \). The key point is to derive a uniform upper bound at points in \(L_j \cap \partial (\epsilon B_i)\). We can reduce to the model situation that we study in the following section.
\end{proof}

\subsection{Upper bound on the bisectional curvature}
On \(\mathbb{C}^2\) we consider the model case \(\omega= i \dd \phi\) where
\begin{equation*}
\phi = \beta^{-2}|z_1|^{2\beta} + \psi (|z_2|^2) |z_1|^2 + |z_2|^2 
\end{equation*}
and \(\psi\) is a standard smooth cut-off function with \(\psi (s) = 0 \) for \( s \leq 1 \) and \(\psi (s) =1\) for \(s \geq 4 \). We shall see that \(\omega\) is indeed positive on \(\{|z_1|< \epsilon\} \times \mathbb{C} \) and has bisectional curvature uniformly bounded above. The metric \(\omega\) interpolates between the flat model \(\omega_{(\beta)}\) (when \(|z_2|<1\)) and the curved metric \(\omega_{euc} + i \dd |z_1|^{2\beta} \) (when \(|z_2|>2\)). We can write \(\omega = \eta + i \dd |z_1|^{2\beta}  \) with \(\eta >0\) in \(|z_2|>1\). It follows from Lemma \ref{CSing} that the bisectional curvature of \(\omega\) is uniformly bounded above on compacts subset of \(|z_2|>1\). The point of this section is to show a uniform upper bound on the bisectional curvature of \(\omega\) around of points \((z_1, z_2)\) with \(z_1 =0\) and \(|z_2|=1\), where Lemma \ref{CSing} does not apply. The details are as follows.

The cut-off function \(\psi\) satisfies \(\psi (s) = 0 \) for \( s \leq 1 \) and \(\psi (s) =1\) for \(s \geq 4 \). Moreover, we have \( 0 < \psi (s) < 1 \) and \(0< \psi'(s)\) if \(1<s<4\). There is some constant \(C>0\) such that \( (\psi')^2 \leq C \psi \), indeed in a neighbourhood of \(1\) we write \(s=1+x\) and we can assume that \(\psi= e^{-1/x}\) for \(x>0\) so the bound follows easily.  

On \(\mathbb{C}^2\) we consider the potential
\begin{equation}
\tilde{\phi} =  \psi (|z_2|^2) |z_1|^2 + |z_2|^2
\end{equation}
Let \( \eta = i \dd \tilde{\phi} = \eta_{i\bar{j}} dz_i d\bar{z}_j \)  and write \(s=|z_2|^2\) so that
\begin{equation}
\eta_{1\bar{1}} =  \psi (s), \hspace{2mm} \eta_{1\bar{2}} = \psi'(s) z_2 \bar{z}_1, \hspace{2mm} \eta_{2\bar{2}} = 1 + |z_1|^2 f(s) 
\end{equation}
where \(f(s) = \psi'(s) + s \psi''(s) = (s \psi')' \). We see that
\( \det (\eta_{i\bar{j}}) = \psi + |z_1|^2 f \psi - (\psi')^2 s |z_1|^2 \). We take \( \epsilon >0 \) so that \( \epsilon^2 |f| < 1/4 \) and \( \epsilon^2 s (\psi')^2 \leq \psi /4 \), therefore 
\begin{equation}\label{eq:deteta}
\det \eta \geq \psi/2 	
\end{equation}
on \( R= \{|z_1|<\epsilon\} \times \mathbb{C} \) and \(\eta>0\) on \(U = R \cap \{|z_2|>1\} \). Note also that \(\eta \geq 0 \) in \(U\), it agrees with \(|dz_2|^2\) if \(|z_2|<1\) and with the euclidean metric if \(|z_2|>2\).

Next we consider
\begin{equation}
\phi = \beta^{-2} |z_1|^{2\beta} + \psi (|z_2|^2) |z_1|^2 + |z_2|^2 ,
\end{equation}
\( \omega = i \dd \phi = g_{i\bar{j}} dz_i d\bar{z}_j = |z_1|^{2\beta-2} idz_1 d\bar{z}_1 + \eta \),  so that
\begin{equation} \label{coef model case}
g_{1\bar{1}} = |z_1|^{2\beta-2} + \psi (s), \hspace{2mm} g_{1\bar{2}} = \psi'(s) z_2 \bar{z}_1, \hspace{2mm} g_{2\bar{2}} = 1 + |z_1|^2 f(s) 
\end{equation}
and
\begin{equation} \label{vol model case}
\det g_{i\bar{j}} = |z_1|^{2\beta-2} (1+ |z_1|^2 f(s) + |z_1|^{2-2\beta} \det \eta ) .
\end{equation}

The only unbounded derivatives of the metric are the ones with all indices equal to \(1\).  Note that  \(g^{1\bar{1}} = O(|z_1|^{2-2\beta})\),  \(g^{1\bar{2}} = O(|z_1|^{3-2\beta})\),  \( g^{2\bar{2}} = O(1) \) and 
\begin{equation}
g_{1\bar{1}, 1} = (\beta-1)z_1^{\beta-2}\bar{z}_{1}^{\beta-1}, \hspace{2mm} g_{1\bar{1}, 1\bar{1}} = (\beta-1)^2 |z_1|^{2\beta-4}
\end{equation}

In \( R \setminus U \) we have \(\omega= \beta^{-2} \omega_{(\beta)} \) and in \(U\) we can write \( \omega = \imath^{*} \mu \) with \( \mu = \eta + i \dd (\beta^{-2}|z_3|^2) \) where \( \imath (z_1, z_2) = (z_1, z_2, z_1^{\beta}) \) and \(\mu\) is a K\"ahler form in \(U \times \mathbb{C} \). We have uniform upper bounds on the bisectional curvature of \(\omega\) over compact subsets of \(U\). We want to show that the bisectional curvature of \(\omega\) is uniformly bounded above in a neighbourhood of \( \{|z_2|=1\} \cap R \). 

Let \(u,v\) be unit tangent vectors in \(T_p \mathbb{C}^2\). Since the metric \(g\) is flat on the zero set of \(\psi\), we can assume that \(\psi(p)>0\). Write \(u = u^1 \partial_1 + u^2 \partial_2\) and \(v=v^1 \partial_1 + v^2 \partial_2\) we want a uniform upper bound on
\begin{equation}\label{eq:Buv}
\mbox{Bisec}(u,v) = R(u, \bar{u}, v, \bar{v}) = \sum R_{i\bar{j}k\bar{l}} u^i \bar{u}^j v^k \bar{v}^l .
\end{equation}
Note that, since the vectors \(u\) and \(v\) have unit length, we have
\begin{equation}\label{eq:compuv}
u^1, v^1 = O(|z_1|^{1-\beta}) \hspace{2mm} \mbox{ and } \hspace{2mm} u^2, v^2 = O(1) .
\end{equation}

Recall that
\begin{equation*}
R_{i\bar{j} k\bar{l}} = - g_{i\bar{j}, k\bar{l}} + g^{p\bar{q}} g_{i\bar{q}, k} g_{p\bar{j}, \bar{l}} 
\end{equation*}
and re-write Equation \eqref{eq:Buv} as
\begin{equation}
B(u, v) = \sum \Lambda_{i\bar{j} k\bar{l}} + \sum \Pi_{i\bar{j} k\bar{l}}
\end{equation}
where \(\Lambda_{i\bar{j} k\bar{l}} = -g_{i\bar{j}, k\bar{l}} u^i \bar{u}^j v^k \bar{v}^l\) and \(\Pi_{i\bar{j} k\bar{l}} =
g^{p\bar{q}} g_{i\bar{q}, k} g_{p\bar{j}, \bar{l}}  u^i \bar{u}^j v^k \bar{v}^l\) (no summations on the \(i,j,k,l\) indices).
The upper bound on the bisectional curvature follows from the next two claims.

\begin{claim}
	\begin{equation}
	\sum \Lambda_{i\bar{j} k\bar{l}} = O(1) - |u^1|^2 |v^1|^2  (\beta-1)^2 |z_1|^{2\beta-4}
	\end{equation}
\end{claim}

\begin{proof}
	The only unbounded derivative \(g_{i\bar{j} k\bar{l}}\) occurs when all indices are equal to \(1\), in which case
	\(g_{1\bar{1}, 1\bar{1}} = -(\beta-1)^2 |z_1|^{2\beta-4}\). The statement follows from this together with Equation \eqref{eq:compuv} which implies that the components \(u^1, v^i\) are bounded.
\end{proof}

\begin{claim}
	\begin{equation}
	\sum \Pi_{i\bar{j} k\bar{l}} = O(1) + |u^1|^2 |v^1|^2  (\beta-1)^2 |z_1|^{2\beta-4}
	\end{equation}
\end{claim}

\begin{proof}
	We use a trick taken from the proof of Lemma A.3 in \cite{JMR}.
	Consider the non-negative Hermitian bilinear form on tensors satisfying \(a_{i\bar{j}k} = a_{k\bar{j}i}\) given by
	\[\langle [a_{i\bar{j}k}], [b_{i\bar{j}k}] \rangle = \sum_{i,j,k,p,q,r} g^{q\bar{j}} (u^i a_{i\bar{j}k} v^k) (\overline{u^p b_{p\bar{q}r}v^r}) \]
	and write \(\| \cdot \|\) for the associated norm. With this definition we have \( \sum \Pi_{i\bar{j} k\bar{l}} = \| [g_{i\bar{j}, k}] \|^2\).
	Recall that 
	\[g_{1\bar{1}, 1} = (\beta-1) |z_1|^{2\beta-4} \bar{z}_1\] 
	is the only unbounded derivative among \(g_{i\bar{j}, k}\). Write \([g_{i\bar{j}, k}] = [A_{i\bar{j}k}] + [E_{i\bar{j}k}]\) with
	\[ A_{i\bar{j}k} = \begin{cases}
	g_{i\bar{j},k}  &\mbox{ if } \hspace{2mm} (i,j,k) \neq (1,1,1) \\
	0  &\mbox{ if } \hspace{2mm} (i,j,k)=(1,1,1)
	\end{cases}
	, 
	\hspace{3mm} E_{i\bar{j}k} = \begin{cases}
	0  &\mbox{ if } \hspace{2mm} (i,j,k) \neq (1,1,1) \\
	g_{1\bar{1}, 1}  &\mbox{ if } \hspace{2mm} (i,j,k)=(1,1,1)
	\end{cases} . \]
	We use the identity
	\begin{equation}\label{eq:trick}
	\| A + E \|^2 = (1+ \delta^{-1}) \|A\|^2 + (1+\delta)\|E\|^2
	\end{equation}
	with \(\delta>0\) to be determined later. 
	
	Note that \(\|A\|^2=O(1)\), as follows from \(g_{i\bar{j},k} = O(1)\) for \((i,j,k)\neq(1,1,1)\) together with the fact that \(g^{i\bar{j}}\) and the components \(u^i, v^i\) are bounded. On the other hand
	\begin{equation}\label{eq:Eterm}
	\|E\|^2 = (\beta-1)^2 |u^1|^2 |v^1|^2 g^{1\bar{1}} |z_1|^{4\beta-6} .
	\end{equation}
	We use Equations \eqref{eq:deteta} and \eqref{vol model case} to bound \(\det g \geq |z_1|^{2\beta-2} (g_{2\bar{2}} + |z_1|^{2-2\beta} \psi/2)\). Since \(g^{1\bar{1}} = (\det g)^{-1} g_{2\bar{2}}\) we obtain
	\begin{equation}\label{eq:g11}
	g^{1\bar{1}} \leq |z_1|^{2-2\beta} \frac{1}{1 + C\psi(p)|z_1|^{2-2\beta}} .
	\end{equation}
	Take \(\delta = C \psi(p) |z_1|^{2-2\beta}\).
	It follows from Equations \eqref{eq:trick}, \eqref{eq:Eterm} and \eqref{eq:g11} that to establish the claim it suffices to bound
	\begin{equation}\label{eq:Abound}
	\psi^{-1}|z_1|^{2\beta-2}\|A\|^2 = \psi^{-1}|z_1|^{2\beta-2} \sum_{i,j,k,p,q,r} g_{i\bar{j}, k} g_{q\bar{p}, \bar{r}} g^{q\bar{j}} u^i \bar{u}^p v^k \bar{v}^r
	\end{equation}
	where the sum runs over indices with \((i,j,k)\neq(1,1,1)\) and \((p,q,r)\neq(1,1,1)\). The terms \(g_{i\bar{j}, k} g_{q\bar{p}, \bar{r}} g^{q\bar{j}}\) contain factors which are quadratic in \(\psi'\) and \(\psi''\). There is a uniform \(C>0\) such that \((\psi')^2\) together with \(\psi' \psi''\) and \((\psi'')^2\)  are \(\leq C \psi\), indeed we can assume that \( \psi = e^{-1/x}\) for \(0<x=s-1 \ll 1 \) and directly check these bounds. This shows that the factor \(\psi^{-1}\) in front of the sum in Equation \eqref{eq:Abound} causes no problem. On the other hand,
	\begin{itemize}
		\item if \((q,j) \neq (2,2)\) then \(g^{q\bar{j}}=O(|z_1|^{2-2\beta})\);
		\item if at least two indices among \(i,p,k,r\) are \(= 1\) then \(u^i \bar{u}^p v^k \bar{v}^r=O(|z_1|^{2-2\beta})\).
	\end{itemize}
	It follows from the previous two bullets that it suffices to consider the cases \((q,j)=(2,2)\) and \((i,p,k,r)=(2,2,2,2)\) or a cyclic permutation of \((1,2,2,2)\). In any case, we will have either \((i,j,k)=(2,2,2)\) or \((p,q,r)=(2,2,2)\) and the bound follows from \(g_{2\bar{2}, 2} = O(|z_1|^2)\). 
\end{proof}

\begin{remark}
	The sectional curvature \( g_{1\bar{1}}^{-2} R_{1\bar{1}1\bar{1}}\) is negative at points where \(\psi>0\), it grows as \(\sim - (\beta-1)^2 \psi |z_1|^{2-4\beta}\) and \(|\mbox{Rm}| = O(\rho^{1/\beta-2})\) with \(\rho=|z_1|^{\beta}\)  is unbounded below if \(\beta>1/2\). 
\end{remark}

Finally we consider the case of our reference metric \(\omega = \omega_{ref} \) given by Equation \eqref{ref met definition}. We have uniform upper bounds on the bisectional curvature on compact subsets of \(\{\eta_0 >0\}\). Since \(\omega\) is isometric to a PK cone in the interior of \(\{\eta_0 = 0\}\) it is enough to obtain a uniform upper bound in neighbourhoods of points \(p \in L_j \cap \partial (\epsilon B_i) \). In a small ball around \(p\) we have \(\omega = \eta_0 + \omega_{F_i}\). We take holomorphic coordinates with \(p \in \{z_1 =0\} \) such that \(\omega_F = |z_1|^{2\beta-2} i dz_1 d \bar{z}_1 + i dz_2 d \bar{z}_2 \). The coefficients of \(\omega\) are
\begin{equation*}
g_{1\bar{1}} = |z_1|^{2\beta-2} + \eta_{1\bar{1}}, \hspace{2mm} g_{1\bar{2}} = \eta_{1\bar{2}}, \hspace{2mm} g_{2\bar{2}} = 1 + \eta_{2\bar{2}}
\end{equation*}  
where \(\eta_{i\bar{j}}\) are the coefficients of the smooth \((1, 1)\)-form \(\eta_0\). We recall that \(\eta_0 \geq 0\), \(\eta_0 \equiv 0\) on \(\epsilon B\), \(\eta_0 > 0\) on the complement of \(\epsilon \bar{B}\) and \(p \in \partial (\epsilon \bar{B}) \). We want to bound above the bisectional curvature \(\mbox{Bisec}(u,v)\) for unit tangent vectors as above.  Same as before the only unbounded derivatives of the metric are the ones in which all indices are equal to \(1\). Moreover, 
\begin{equation}
\det (g) = |z_1|^{2\beta-2}(1 + \eta_{2\bar{2}} + \det \eta_0)
\end{equation}
and we have  \(g^{1\bar{1}} = O(|z_1|^{2-2\beta})\),  \(g^{1\bar{2}} = O(|z_1|^{3-2\beta})\),  \( g^{2\bar{2}} = O(1) \).
The bound on  \(\mbox{Bisec}(u,v)\) is obtained by minor modifications of the previous model case.

\subsection{Ricci potential}
Let \((\xi_1, \xi_2)\) be affine coordinates in $ U \subset \mathbb{CP}^2$ with its line at infinity \( L_{\infty} \) different from any of the lines \(L_j\) from the arrangement, then we can write \(L_j  = \{\ell_j =0\} \) over \(U\) for some non-zero affine functions \(\ell_j(\xi_1, \xi_2)\) and \(j=1, \ldots, n\). Similar to Appendix \ref{pm sphere section}, we introduce the locally defined holomorphic volume form
\begin{equation}
	\Omega = \ell_1^{\beta_1-1} \ldots \ell_n^{\beta_n-1} d\xi_1 d\xi_2 .
\end{equation}
The requirement that \( \sum_{j=1}^{n} (1-\beta_j) = 3 \) implies that \(\Omega\) extends to \(L_{\infty}\) as a locally defined defined holomorphic section \(\Omega_{\infty}\) of \(\mathcal{O}(-3)|_{L_{\infty}}\). At points of intersection of \(L_j\) with \(L_{\infty}\) we can contract \(\Omega_{\infty}\) with a local non-vanishing section \(\mathbf{n}\) of the normal bundle of \(L_{\infty}\) to obtain a \(1\)-form \(\Omega \otimes \mathbf{n}\) on \(L_{\infty}\) with a logarithmic singularity. If we further assume that \(L_{\infty}\) doesn't go through any of the multiple points \(\{x_j\}\) then around each \(p_j = L_j \cap L_{\infty}\) we can write \(\Omega \otimes \mathbf{n} =  \xi^{\beta_j-1} d\xi \) where \(\xi\) is a local coordinate of \(L_{\infty}\) centred at \(p_j\).
	
\begin{lemma} \label{ricci lemma}
		\begin{equation} \label{vol ref met}
			\omega_{ref}^2 = e^{-f} \Omega \wedge \overline{\Omega}
		\end{equation}
with \(f \in C^{\alpha}_{loc}(X') \). Around \(x_i\) we have \( f = f_{x_i} + h_i \) with \(f_{x_i} \in \mathbb{R} \), \(h_i(0)=0\) and \(  \dd h_i =0 \).
\end{lemma}
The lemma follows immediately from the fact that \(\omega_{ref}\) has standard conical singularities along  \(L_j^{\times}\) and that it agrees with a PK cone metric around each \(x_i\) (in particular is Ricci flat in such neighbourhoods). It follows from Equation \eqref{vol ref met} that \(\mbox{Ric}(\omega_{ref}) = i \dd f\).

\subsection{Sobolev inequality}

Let \(r\) be a continuous positive function on \(\mathbb{CP}^2 \setminus \{x_i\}\) that agrees with the radius functions of the corresponding PK cones \(g_{F_i}\) in small neighbourhoods of the multiple points \(x_i\), see Section \ref{sect:refmet} and Equation \eqref{ref met definition}. 
Let us denote \(X= \mathbb{CP}^2\) and \(X' = \mathbb{CP}^2 \setminus \cup_i \{x_i\}\).
Our goal is to establish the next.

\begin{proposition}\label{prop:Sob}
	Let \(g_{ref}\) be the reference metric given  by Equation \eqref{ref met definition}. Then there is some \(c>0\) such that the following inequality holds
	\begin{equation} \label{Sob ineq}
	\left( \int_X |u|^{2\alpha} \right)^{1/\alpha} \leq c \int_X | \nabla u|^2
	\end{equation}
	for all \( 1 \leq \alpha \leq 2 \) and \( u \in C^{1}_{loc}(X') \) with \(u/r \in L^2 \) at the \( \{x_i\} \) and \( \int_X u =0 \). 
	(The quantities in Equation \eqref{Sob ineq} are measured with respect to \(g_{ref}\).)
\end{proposition}

\begin{proof}
	Suppose there is a homeomorphism \(\Psi\) of \(\mathbb{C P}^2\) which is differentiable on the complement of \(\cup_j L_j\) and such \(\Psi^*g_{ref}\) is uniformly equivalent to \(g_{ref}\). Then if Equation \eqref{Sob ineq} holds for \(\Psi^{*}g_{ref}\) and \(r \circ \Psi\), the statement of Proposition \ref{prop:Sob} follows after  pulling-back by \(\Psi\). On the other hand, by Proposition 2.5 in \cite{heinsun}, Equation \eqref{Sob ineq} holds in the setting of conifolds and in particular in the case of smooth Riemannian metrics. We conclude that, in order to prove Proposition \ref{prop:Sob}, it is enough to construct \(\Psi\) as above such that \(\Psi^*g_{ref}\) is quasi-isometric to a smooth Riemannian metric.
	
	First we note that any PK cone $g_F$ is up to a diffeomorphism quasi-isometric to the euclidean \(\mathbb{R}^4\), that is there exists a diffeomorphism $P$ of the complement of the conical lines such that   $\Lambda^{-1}g_{euc} \leq P^{*}g_F \leq \Lambda g_{euc}$ for some $\Lambda >1$. This is indeed clear from the construction of $g_F$: The spherical metric with cone singularities on $\mathbb{CP}^1$ is, up to a diffeomorphism of the punctured sphere modelled on \( \rho e^{i\theta} \to \rho^{1/\beta}e^{i\theta} \) around the conical points and supported on a small disc centred at the \(L_i\), quasi-isometric to the round metric and we can lift this to a diffeomorphism \( \tilde{P} \) of \( S^3 \setminus L \) such that  $ \Lambda^{-1} g_{S^3(1)} \leq \tilde{P}^{*} \overline{g} \leq \Lambda g_{S^3(1)} $. Since \( g_F = dr^2 + r^2 \overline{g} \) and \( g_{euc} = dr^2 + r^2 g_{S^3(1)} \), the map \( P (r, \theta) = (r, \tilde{P}(\theta) ) \) satisfies $\Lambda^{-1}g_{euc} \leq P^{*}g_F \leq \Lambda g_{euc}$.
	
	Around points of $L_j^{\times}$ we can find complex coordinates \(z_1, z_2\) such that \(L_j^{\times}=\{z_1=0\}\) and the reference metric \(\omega_{ref}\) is quasi-isometric to the model \(g_{(\beta)}\), therefore is quasi-isometric to a smooth metric in cone coordinates. More precisely, let \( C = \cup_j L_j \) and \( \nu_C \) be its normal bundle together with a continuous Hermitian metric \(h\) smooth on \( C \setminus \cup_i \{x_i\} \). We take \( \epsilon >0 \) and identify \( \{ v \in \nu_C \hspace{1mm} \mbox{s.t.} \hspace{1mm} |v|_h < \epsilon \} \) with a tubular neighbourhood \(C \subset U_{\epsilon} \subset \mathbb{CP}^2\). We let \( Q \) be a diffeomorphism of \( \mathbb{CP}^2 \setminus C \) supported on the complement of \(U_{\epsilon}\) and equal to \( v \to |v|_h^{1/\beta-1} v \) on \(U_{\epsilon/2}\). It is then clear that for any compact \( K \subset \mathbb{CP}^2 \setminus \cup_i \{x_i\} \) there is a constant \( \Lambda \) such that  $\Lambda^{-1}g \leq Q^{*}g_{ref} \leq \Lambda g$  holds over \(K\) for a smooth metric \(g\).
	
	We can patch together the diffeomorphisms \(P\) and \(Q\) to obtain a diffeomorphism \(\Psi\) of \( \mathbb{CP}^2 \setminus \cup_{j=1}^n L_j \) such that \( \Psi^{*}g_{ref} \) is quasi-isometric to a smooth metric. The proposition follows.
\end{proof}

We will use inequality \eqref{Sob ineq} with \(\alpha=1\) (Poincar\'e inequality) and \(\alpha=2\) in deriving the \(C^0\)-estimate.
In general, \(\alpha\) is allowed to vary in the interval \([1, \frac{m}{m-2}]\) where \(m\) is the real dimension.
For the weighted estimates we will use the next version for some \(1<\alpha<2\).

\begin{lemma}
	Let \(g_F\) be a PK cone metric on \(\mathbb{C}^2\) with radius function \(r\) and let \(V=\{r\leq 1\}\). Then there is a constant \(c\) such that
	\begin{equation} \label{Sob ineq2}
	\left( \int_V r^{2\alpha-4} |u|^{2\alpha} \right)^{1/\alpha} \leq c \int_V | \nabla u|^2
	\end{equation}
	for all \( 1 \leq \alpha \leq 2 \) and \( u \in C^{1}_{loc}(V\setminus\{0\})\) with \(u/r \in L^2(V) \)  and \(u|_{\partial V} =0 \). 
\end{lemma}

\begin{proof}
	As shown in the proof of Proposition \ref{prop:Sob}, the PK cone metric on \(\mathbb{C}^2\) is quasi-isometric (after applying a suitable diffoemorphism on the regular part) to the Euclidean metric on \(\mathbb{R}^4\). The statement then follows from the known case for Euclidean space, see Lemma 2.3 in \cite{heinsun}.
\end{proof}

\section{Linear Theory} \label{linear thry sct}

\subsection{Donaldson's interior Schauder estimates}

Consider the model metric $g_{(\beta)} = \beta^2 |z_1|^{2\beta-2} |dz_1|^2 + |dz_2|^2$ on $\mathbb{C}^2$ and write $z_1= \rho^{1/\beta} e^{i\theta}$ as before.
There is an induced distance $d_{\beta}$ and therefore, for each \( \alpha \in (0, 1) \), a H\"older  semi-norm
\begin{equation} \label{Holder seminorm}
[u]_{\alpha} = \sup_{x, y} \frac{|u(x) - u(y)|}{d_{\beta}(x, y)^{\alpha}} 
\end{equation}
on continuous functions defined on domains of \( \mathbb{C}^2 \). In the cone coordinates \( (\rho e^{i\theta}, z_2) \)
\(g_{(\beta)}\) is quasi-isometric to the Euclidean metric and \eqref{Holder seminorm} is equivalent to the standard H\"older semi-norm with respect to the Euclidean distance. 

Set $\epsilon = d\rho + i \beta \rho d\theta$. A $(1,0)$-form $\eta$ is $C^{\alpha}$  if $\eta= u_1 \epsilon +  u_2 dz_2$ with $u_1, u_2$ $C^{\alpha}$ functions in the usual sense in the cone coordinates; it is also required that $u_1  =0$  on the singular set $\lbrace z_1 = 0 \rbrace$.  For \((1,1)\)-forms we use the basis $\lbrace \epsilon \overline{\epsilon}, \epsilon d\overline{z_2}, dz_2\overline{\epsilon}, dz_2 d\overline{z_2} \rbrace$. We say that the \((1, 1)\)-form $\eta$ is $C^{\alpha}$ if its components are $C^{\alpha}$ functions; we also require the components  corresponding to $\epsilon d\overline{z_2},  dz_2 \overline{\epsilon}$ to vanish on \( \{z_1 =0\} \).  
We set $C^{2, \alpha}$ to be the space of $C^{\alpha}$ (real) functions $u$ such that $\partial u, i\partial \overline{\partial} u$ are $C^{\alpha}$, this space of functions is also known as \(C^{2, \alpha, \beta}\) but we clear \(\beta\) from the notation. It is straightforward to introduce norms: We define the $C^{\alpha}$ norm of a function $\|u\|_{\alpha}$ as the sum of its $C^0$ norm $\|u\|_0$ and its $C^{\alpha}$ semi-norm $[u]_{\alpha}$.
The $C^{2, \alpha}$ norm of a function $u$, denoted by \( \|u\|_{2, \alpha} \), is the sum of  $\|u\|_{\alpha}$,  the $C^{\alpha}$ norm of the components of $\partial u$ in  and the $C^{\alpha}$ norm of the components of $i \partial \overline{\partial} u$.

We are interested in the equation $\triangle u =f$ where $\triangle$ is the Laplace operator of $g_{(\beta)}$. We define $L^2_1$ on domains of $\mathbb{C}^2$  by means of the usual norm $ \| u \|_{L^2_1}^2 = \int |\nabla u|^2 + \int u^2$, in the cone coordinates $L^2_1$ coincides with the standard Sobolev space. Let $u$ be a function that is locally in $L^2_1$, we say that $u$ is a weak solution of $\triangle u =f$ if
$$ \int \langle\nabla u, \nabla \phi \rangle = - \int f\phi$$ 
for all smooth compactly supported $\phi$.

Fix $\alpha < \beta^{-1} -1$ and let $u$ be a weak solution of $\triangle u = f$ on $B_2$ with $f \in C^{\alpha}(B_2)$. It is proved in \cite{donaldson1}  that $u \in C^{2, \alpha}(B_1)$ and there is a constant $C$ -independent of $u$- such that 
\begin{equation} \label{SCH}
\|u\|_{C^{2, \alpha}(B_1)} \leq C \left( \|f\|_{C^{ \alpha}(B_2)} + \|u\|_{C^0(B_2)} \right) .
\end{equation}

We mention three differences between this result and the standard Schauder estimates:
\((i)\) We don't have estimates for all the second derivatives of $u$, for example $\partial^2 u / \partial \rho^2$. \((ii)\) If $\triangle u \in C^{\alpha}$ then the component of $\partial u$ corresponding to $\epsilon$ needs to vanish along the singular set (see \cite{donaldson1} Section 4.3). \((iii)\) The estimates require $\alpha < \beta^{-1} -1$.
These differences  can be explained by the fact that if $p$ is a point outside the singular set  and $\Gamma_p = G(\cdot, p)$ is the Green's function for $\triangle$ with a pole at \(p\);  then around points of $\lbrace z_1 =0 \rbrace$ there is a convergent series expansion
\begin{equation} \label{expansion}
\Gamma_p = \sum_{j, k \geq 0} a_{j, k} (z_2) \rho^{(k/\beta) + 2j} \cos (k\theta)
\end{equation}
with $a_{j, k}$ smooth functions given in terms of Bessel's functions.
 
Let \(X\) be a closed complex surface and \(C \subset X \) a smooth complex curve, we define the Banach space \(C^{2, \alpha}\) by means of a finite covering by complex coordinates adapted to \(C\). Let \(g\) be as in Definition \ref{def kmcs}, around each $p \in C$ we can find adapted coordinates $(z_1, z_2)$ such that 
\begin{equation} \label{DEF}
\omega = \omega_{(\beta)} + \eta
\end{equation}
where $\eta \in C^{\alpha}$ and all the coefficients of $\eta$ in the basis  $\lbrace \epsilon \overline{\epsilon}, \epsilon d\overline{w}, dw\overline{\epsilon}, dwd\overline{w} \rbrace$ vanish at $p$. After a dilation and multiplying by a cut-off function we obtain interior estimates as in Equation \eqref{SCH} for \(\Delta_g\). These estimates patch together to give a global parametrix  and \( \Delta_g : C^{2, \alpha}(X) \to C^{\alpha}(X) \) is Fredholm operator of index zero.

\subsection{Weighted H\"older spaces on PK cones} \label{wh}

Let \(C(Y) = \mathbb{R}_{>0} \times Y \cong \mathbb{C}^2 \setminus \{0\} \) be a PK cone with metric $g_F = dr^2 + r^2 \overline{g} $. In global complex coordinates \((z, w)\) its singular set \(S\) consists of a bunch of complex curves (or lines) \( \{z^p = a_j w^q\} \) going through the origin along which \(g_F\) has cone angle \(2\pi\beta_j\). Fix \(0< \alpha < \min_j (1/\beta_j) -1 \). 

Let \(U\) be a bounded domain compactly contained in \(C(Y)\). Around each point  \(p \in U\)  there are complex coordinates $(z_1, z_2)$ in which $g_F = g_{(\beta_j)}$ for some \(j\) if $p \in S \cap U$ and $g_F = g_{euc}$ if $p \notin S$. We fix a finite cover of $U$ by such coordinates and define the spaces $C^{\alpha} (U)$ and $C^{2, \alpha}(U)$ in the obvious way.
Write $B_R = \lbrace r < R \rbrace$ for the metric ball of radius $R$ around the apex.  Consider the nested annuli \( A \subset \tilde{A} \) with $A = B_2 \setminus \overline{B_1}$ and $ \tilde{A} = B_4 \setminus \overline{B_{1/2}}$. It follows from the interior Schauder estimates \eqref{SCH} that there is a constant $C$ such that for every $u \in C^{2, \alpha}(\tilde{A})$
\begin{equation} \label{int estimate}
\|u\|_{C^{2, \alpha}(A)} \leq C \left( \| \Delta u\|_{C^{ \alpha}(\tilde{A})} + \|u\|_{C^0(\tilde{A})} \right) , 
\end{equation}
where $\Delta$ is the Laplace operator of $g_F$.

Let $\mu \in \mathbb{R}$.  For $\lambda > 0$ we denote by  $D_{\lambda}$ the dilation $D_{\lambda} (r, y) = (\lambda r , y)$. For  a continuous function \(f\) on $C(Y)$ we define $ f_{\lambda, \mu} = \lambda^{-\mu} \cdot f_{\lambda}$ with \( f_{\lambda} = f \circ D_{\lambda}\) and we set
\begin{equation} \label{norm1}
\| f \|_{\alpha, \mu} = \sup_{\lambda > 0} \|f_{\lambda, \mu} \|_{C^{\alpha}(A)} .
\end{equation}
$ C^{\alpha}_{\mu}$ is the space of functions for which the norm \eqref{norm1} is finite. If \( f \in C^{\alpha}_{\mu} \)  then $f = O (r^{\mu})$ in the sense that  $|f(x)| \leq M r(x)^{\mu}$  for some constant $M>0$ and all \(r \in (0, + \infty)\). Similarly we define $C^{2, \alpha}_{\mu}$ by the norm
\begin{equation} \label{norm2}
\| u \|_{2, \alpha, \mu} = \sup_{\lambda > 0} \|u_{\lambda, \mu} \|_{C^{2, \alpha}(A)} .
\end{equation}
If \( u \in C^{2, \alpha}_{\mu + 2}\) then $| \partial \overline{\partial} u |_{g_F} = O( r^{\mu})$. As usual  $ C^{\alpha}_{\mu}$ and $ C^{2, \alpha}_{\mu}$ are Banach spaces. 

Since
\begin{equation} \label{lap}
\triangle = \frac{\partial^2}{\partial r^2} + \frac{3}{r} \frac{\partial}{\partial r} + \frac{1}{r^2} \triangle_{\overline{g}} , 
\end{equation}
it follows that $\triangle (u_{\lambda}) = \lambda^2 \cdot (\triangle u)_{\lambda}$ and it is then clear that $\triangle$ is a bounded operator from $C^{2, \alpha}_{\mu + 2}$ to $C^{\alpha}_{\mu}$. The interior Schauder estimates \eqref{int estimate} give us the following:

\begin{lemma} \label{Est 1}
	There is a constant $C = C(\alpha, \mu)$ such that for every $u \in C^{2, \alpha}_{\mu +2}$ with $\triangle u =f$ 
	$$ \|u\|_{2, \alpha, \mu + 2} \leq C \left( \|f\|_{\alpha, \mu} + \| u \|_{0, \mu +2} \right) . $$
\end{lemma}

\begin{proof}
	Apply the interior estimate \eqref{int estimate} to $u_{\lambda, \mu +2} = \lambda^{-\mu -2} u_{\lambda}$ to get
	$$ \|u_{\lambda, \mu +2} \|_{C^{2, \alpha}(A)} \leq C \left( \| \lambda^{-\mu} f_{\lambda} \|_{C^{\alpha}(\tilde{A})} + \| \lambda^{-\mu -2} u_{\lambda} \|_{C^0(\tilde{A})} \right).$$
	It is clear that we can replace \(A\) with \(\tilde{A}\) in Equation \eqref{norm1} to get an equivalent norm,  therefore the first term on the right hand side is bounded by $\|f\|_{\alpha, \mu}$ and the second term is bounded by $\|u\|_{0, \mu +2}$.
\end{proof}

\begin{remark}
	In fact we have proved that if $u$ is locally in $C^{2, \alpha}$, $\triangle u \in C^{\alpha}_{\mu}$ and $\| u \|_{0, \mu +2}$ is finite, then $u \in C^{2, \alpha}_{\mu + 2}$ and the above estimate holds.
\end{remark}

Our next goal is to bound $\|u\|_{0, \mu +2}$ in terms of $\|f\|_{\alpha, \mu}$ and this holds provided that $\mu + 2$ does not belong to the discrete set of `Indicial Roots'. In order to define this set we digress a bit and discuss some basics of spectral theory for the Laplace operator  $\triangle_{\overline{g}}$ of the singular metric on the 3-sphere. 

On $(S^3, \overline{g})$ we have  Hilbert spaces \(L^2\) and \(L^2_1\) equipped with norms  $ \|f\|_{L^2}^2 = \int f^2$ and $\| u \|_{L^2_1}^2 = \int u^2 + \int | \nabla u |^2 $. Up to a diffeomorphism  $ \overline{g}$ is quasi-isometric to a smooth metric and $L^2$, $L^2_1$ correspond to the standard spaces. In particular we have a compact inclusion $L^2_1 \subset L^2$. A function
$ f \in L^2$ defines a bounded linear functional $T$ on $L^2_1$ by $ T(\phi) = \int f\phi$ and if $u$ is such that $ T = \langle u, - \rangle_{L_1^2}$ then $u$ is said to be a weak solution of $-\triangle_{\overline{g}} u + u = f$. The map $K(f) = u$ is a bounded linear map between $L^2$ and $L^2_1$, composing this map with the compact inclusion we have a map $K: L^2 \to L^2$ which is compact and self-adjoint. It follows from the spectral theorem that we can find an orthonormal basis $\lbrace \phi_i \rbrace_{i\geq 0}$ of $L^2$ such that $K(\phi_i) = s_i \phi_i$ and $s_i \to 0$. Unwinding the definitions we get that $ \triangle_{\overline{g}} \phi_i = - \lambda_i \phi_i$ with $0= \lambda_0 \leq \lambda_1 \leq \lambda_2 \leq \ldots $ and $\lambda_i = (1-s_i)/s_i \to \infty$.
For each $\lambda_i$ define $ \delta_i^{\pm}$ to be the solutions of $ \lbrace s(s+2) = \lambda_i \rbrace$ with $\delta_i^{+} \geq 0$  and $\delta_{i}^{-} \leq -2$. The set of Indicial Roots is $ \mathcal{I} = \lbrace \delta_i^{\pm} , i \geq 0 \rbrace$ and is a discrete set of real numbers symmetrically distributed around \(-1\). A function \(u\) on \(C(Y)\) is said to be \(\delta\)-homogeneous if \( u(r, y) = r^{\delta} \phi(y)\) and it is clear from Equation \eqref{lap} that such a function satisfies \(\Delta u =0\) if and only if \(\Delta_{\overline{g}} \phi = -\lambda \phi \) with \( \lambda = \delta (\delta+2) \). It is then clear that
\begin{equation} \label{indicial roots}
	\mathcal{I} = \{ \delta : \hspace{2mm} \mbox{there is a non-zero \(\delta\)-homogeneous function \(u\) with } \Delta u = 0 \}
\end{equation}

\begin{lemma} \label{inj}
	If $ u \in C^{2, \alpha}_{\delta} $ is such that $\triangle u =0$ and  $ \delta \notin I$ then $u=0$. 
\end{lemma}

\begin{proof}
	Write $ u (r, \theta) = \sum_{i=0}^{\infty} u_i (r) \phi_i (\theta)$, where $ u_i (r) = \int_{S^3} u(r, \cdot) \phi_i$. It follows from H\"older's inequality that if $ |u| \leq C r^{\delta}$ then $ |u_i (r) | \leq C (\mbox{Vol} (\overline{g}))^{1/2} r^{\delta}$. On the other hand the equation $\triangle u =0$ implies 
	$$ u_i'' + \frac{3}{r} u_i' - \frac{\lambda_i} { r^2} u_i =0 , $$
	therefore $u_i = A r^{\delta_i^{+}} + B  r^{\delta_i^{-}}$  for some constants $A$ and $B$. Since $\delta \not= \delta_i^{\pm}$ must have $u_i=0$ for all \(i\).
\end{proof}

\begin{proposition} \label{estimate cone}
	If $ \mu + 2 \notin I$ then there is $C>0$ such that
	for every $u \in C^{2, \alpha}_{\mu +2}$ with $ \triangle u = f$ 
	\begin{equation} \label{Est 2}
	\|u\|_{2, \alpha, \mu +2} \leq C \|f\|_{\alpha, \mu} .
	\end{equation}
	 
\end{proposition}

\begin{proof}
	Write \(\mu + 2 = \delta\). If the result is not true then we would be able to take a sequence $\{ u_k \}$ with $ \| u_k \|_{2, \alpha, \delta} = 1$, $ \triangle u_k = f_k$ and $\| f_k \|_{\alpha, \delta -2} \to 0$. It follows from Lemma \ref{Est 1} that $\|u_k\|_{0, \delta} \geq 2 \epsilon$ for some $\epsilon > 0$. Hence we can find $x_k$ such that $ r(x_k)^{-\delta} | u_k (x_k) | \geq \epsilon$. Consider the sequence $\tilde{u_k} = (u_k)_{L_k, \delta}$ where $L_k = r(x_k)$. Write $ x_k = (r(x_k), \theta_k)$,  then $| \tilde{u_k} (\tilde{x_k}) | \geq \epsilon$ with $\tilde{x_k} = (1, \theta_k)$. On the other hand $ \tilde{f_k} = \triangle \tilde{u_k} =  L_k^{-\delta + 2} (f_k)_{L_k} = (f_k)_{L_k, \mu}$, with $\mu = \delta -2$. The key point is that $ \| u \|_{2, \alpha, \delta} = \| u_{L, \delta} \|_{2, \alpha, \delta}$ and $ \| f \|_{\alpha, \mu} = \| f_{L, \mu} \|_{\alpha, \mu}$ for any $L >0$ and $f,u$ any functions. So that $ \| \tilde{u_k} \|_{2, \alpha, \delta} = 1$ and $\| \tilde{f_k} \|_{\alpha, \delta -2} \to 0$. Let $K_n = \overline{B_n} \setminus B_{1/n}$ for $n$ an integer $\geq 2$. Arzela-Ascoli and the bound $ \| \tilde{u_k} \|_{2, \alpha, \delta} = 1$ imply that we can take a subsequence $ \tilde{u_k}^{(n)}$ which converges in $C^2(K_n)$ to some function $u_n$ such that $\triangle u_n =0$. The diagonal subsequence $ \tilde{u_n}^{(n)}$ converges to a function $u$ in $ \mathbb{C}^2 \setminus \lbrace 0 \rbrace$ which is in $C^2_{\delta}$ and $\triangle u =0$. Since $| \tilde{u_k} (\tilde{x_k})| \geq \epsilon$ we see that $u \not= 0$, but this contradicts Lemma \ref{inj}
\end{proof}

It follows from Lemma \ref{inj} that if \(\mu+2 \notin I \) then \( \Delta : C^{2, \alpha}_{\mu +2} \to C^{\alpha}_{\mu} \) is injective and Proposition \ref{estimate cone} implies that its image is closed; indeed this map is an isomorphism (see \cite{Bartnik}) but we will not prove this as we won't need this result.

\subsection{Homogeneous harmonic functions on PK cones (\cite{heinsun, CH})} Let \((M, g)\) be a  Riemannian manifold with \(\mbox{Ric}(g) \geq (m-1)g \) where
\(m\) is the real diemnsion of \(M\), the Lichnerowicz-Obata theorem asserts that the first non-zero eigenvalue of the Laplace operator satisfies \(\lambda_1(M) \geq m \) and equality holds if and only if \((M, g)\) is isometric to the round sphere of radius \(1\). An immediate consequence is that if \(u\) is a non-zero \(\delta\)-homogeneous harmonic function on a Riemannian cone with non-negative Ricci curvature then \(\delta \geq 1\) and equality holds if and only if the cone is isometric to Euclidean space. 

It is proved in \cite{CH, heinsun} that any \(\delta\)-homogeneous harmonic function with \( 1 < \delta < 2 \) on a Riemannian K\"ahler cone of non-negative Ricci curvature is pluri-harmonic. If \( \delta =2 \) then \(u=u_1+u_2\) with \(\dd u_1=0\) and \( \xi (u_2) =0 \), that is \( u_2 \) is \(\xi\)-invariant where \(\xi = r I \partial_r \) is the Reeb vector field. If the K\"ahler cone is in addition Ricci-flat then the space of holomorphic vector fields that commute with \( r \partial_r \) can be written as \footnote{This decomposition can be regarded as an extension of the Matsushima theorem on reductivity of the automorphism group of  KE Fano varieties.} \( \mathfrak{p} \oplus I \mathfrak{p} \) where \(\mathfrak{p}\) is spanned by \(r\partial_r\) and the gradient vector fields of the \(\xi\)-invariant  \(2\)-homogeneous harmonic functions, moreover all elements in \(I\mathfrak{p}\) are Killing vector fields. 
The proofs of these facts extend in a straightforward manner to the case of PK cones and we record them into the following
\begin{lemma} \label{hom harm lemma}
	Let \(u\) be a \(\mu\)-homogeneous harmonic function on \(g_F\). If \(\mu >0 \) then \(\mu \geq 1 \) and \( \mu =1 \) only if \(g_F\) is the Euclidean metric.  If \(1 < \mu < 2 \) then \(u\) is pluri-harmonic. 
\end{lemma}

If \(\mu=2\) then \(u=u_1+u_2\) with \(\dd u_1=0\) and \(u_2\) is \(\xi\)-invariant. The gradient of \(u_2\) is an holomorphic vector field which commutes with \(r \partial_r \) and whose flow preserves the conical set, if the PK cone is not a product then  \(r\partial_r\) is the only vector field with those properties so we must have \(u_2=0\) and there are no \(2\)-homogeneous \(\xi\)-invariant harmonic functions. In the product case one must take into account dilations of one of the factors as we explain next.

Let \(g_F\) be the PK cone \(\mathbb{C}_{\beta_1} \times \mathbb{C}_{\beta_2} \), so that \(r^2 = |z_1|^{2\beta_1} + |z_2|^{2\beta_2} \). Expressions are simpler if we work in the holomorphic tangent bundle, so we write \(\Xi = (r\partial_r)^{1,0} = (1/2) (r\partial_r -i \xi) = \nabla^{1, 0} r^2 \) where \(\nabla^{1, 0}\) denotes the \((1, 0)\)-component of the gradient with respect to \(g_F\) and therefore \( \Xi = (1/\beta_1) z_1 \partial_{z_1} + (1/\beta_2) z_2 \partial_{z_2} \). The space \(\mathfrak{h}\) of holomorphic vector fields which commute with \(\Xi\) is the \(\mathbb{C}\)-linear span of \(\partial_{z_1}\) and \(\partial_{z_2}\). Fitting with our previous discussion, we have \( \mathfrak{h} = \mbox{span}_{\mathbb{C}} \{\Xi, \nabla^{1, 0} h\} \) where
\begin{equation}
	h = |z_1|^{2\beta_1} - |z_2|^{2\beta_2}
\end{equation}
is (the unique up to a constant multiple) \(2\)-homogeneous \(\xi\)-invariant harmonic function with respect to \(g_F\) and \(\nabla^{1, 0} h = (1/\beta_1) z_1 \partial_{z_1} - (1/\beta_2) z_2 \partial_{z_2} \).

\begin{remark}
	The space of homogeneous harmonic functions on \(\mathbb{C}_{\beta_1} \times \mathbb{C}_{\beta_2}\) with sub-quadratic growth is explicitly identified in \cite[Proposition 3.4]{dBEsch} by using separation of variables. On the other hand, after the first version of this paper appeared, Donladson's Schauder estimates were extended to the normal crossing situation by Guo-Song \cite{GuoSong}. Using Guo-Song's estimates we could avoid introducing a weight function at the double points of the  arrangement and instead deal with them in the same way as we do with points in \(L_j^{\times}\) (i.e. smooth points of the arrangement).
\end{remark}

We introduce the subset \( \mathcal{H} \subset \mathcal{I} \) of homogeneous pluri-harmonic functions
\begin{equation*} 
\mathcal{H} = \{ \delta : \hspace{2mm} \mbox{there is a non-zero \(\delta\)-homogeneous function \(u\) with } \dd u = 0 \} .
\end{equation*}
For any \(u \in \mathcal{H} \) there is a corresponding \(v\) such that \( u + i v \) is holomorphic \(\delta\)-homogeneous, so we refer to \(\mathcal{H}\) as the holomorphic spectrum. It follows from Hartog's theorem that if \(\delta \in \mathcal{H}\) then \(\delta \geq 0 \). In the case of PK cones the holomorphic spectrum is explicit; indeed Liouville's theorem implies that any \(\delta\)-homogeneous holomorphic function must be a polynomial of the complex coordinates \(z, w\). In the regular case we must have \( \mathcal{H} = \{ m/\gamma : \hspace{2mm} m \in \mathbb{N}_0 \} \) and in the case of \( \mathbb{C}_{\beta_1} \times \mathbb{C}_{\beta_2} \) we get \( \mathcal{H} = \{ m/\beta_1 + n/\beta_2 : \hspace{2mm} m, n \in \mathbb{N}_0 \} \).

\subsection{Main result}

Let \( \beta = \max \beta_j \) and take \( 0 < \alpha < \beta^{-1} -1 \). Write \( \Delta \) for the Laplace operator of \(g_{ref}\). For each ball \( B \subset X' \) we have Donaldson's interior Schauder estimates
\begin{equation} \label{DiSe}
	 \| u \|_{2, \alpha; \frac{1}{2}B} \leq c \left( \| \Delta u \|_{\alpha; B} + \| u \|_{C^0(B)} \right).
\end{equation}
The Sobolev inequality together with elliptic regularity imply that for any \( f \in C^{\alpha}_{loc}(X') \cap L^{\infty}(X) \) with \( \int_X f =0 \) there is a unique \( u \in C^{2, \alpha}_{loc}(X') \cap L^{\infty}(X) \) with \( \int_X u =0 \) such that \( \Delta u =f \).

Let \(\Phi_i\) be affine coordinates centred at the multiple points. We use the maps \(\Phi_i\) (see section \ref{reference metrics}) to identify a neighbourhood of a multiple point \(x_i\) with a ball \(B_i= \{ r < 1 \} \) in the corresponding PK cone \(g_{F_i}\), we can assume that \(3B_i\) are pairwise disjoint. As in section \ref{reference metrics} we let \(\chi\) be a smooth function on \(\mathbb{CP}^2\) with values in the interval \([0, 1]\) with \(\chi \equiv 0 \) on \( \overline{\cup_i B_i} \), \(\chi>0\) on \( \mathbb{CP}^2 \setminus \overline{\cup_i B_i} \)  and \(\chi \equiv 1 \) on \(\mathbb{CP}^2 \setminus \cup_i 2B_i\). Let \(N\) be the number of multiple points. For \( \mu = (\mu_1, \ldots, \mu_N) \in \mathbb{R}^N \) we define
\begin{equation} \label{final wh norm}
	\|u\|_{2, \alpha, \mu} = \|\chi u\|_{2, \alpha} + \sum_i \|(1-\chi)u \circ \Phi_i \|_{2, \alpha, \mu_i}
\end{equation}
The \(\|(1-\chi)u \circ \Phi_i \|_{2, \alpha, \mu_i}\) terms stand for the weighted H\"older norms on the PK cones introduced in \ref{wh}. 
The norm in Equation \eqref{final wh norm} defines the weighted H\"older space \(C^{2, \alpha}_{\mu}\). Similarly there is the space \(C^{\alpha}_{\mu}\). If \(\Delta\) denotes the Laplace operator of \(\omega_{ref}\) then \( \Delta : C^{2, \alpha}_{\mu} (X) \to C^{\alpha}_{\mu-2}(X) \) is a bounded linear operator where \(\mu -2 = (\mu_1 -2, \ldots, \mu_N -2) \). 

For each multiple point \(x_i\) we have \( g_{F_i} = dr^2 + r^2 \overline{g}_i \). Write \( \mathcal{E}_i = \{ 0= \lambda_0^{(i)} < \lambda_1^{(i)} < \ldots \} \) for the spectrum of the Laplacian of \(\overline{g}_i\). Then \( \mu^{(i)}_{j, \pm} = -1 \pm \sqrt{1 + \lambda^{(i)}_j} \) is  the homogeneity of the corresponding harmonic functions on \(g_{F_i}\) and \( \mathcal{D}_i = \{ \mu^{(i)}_{j, \pm} \}_{j=1}^{\infty} \) is the set of indicial roots. Define \( I_i : \mathbb{R} \to \mathbb{Z} \) as 
\begin{equation*}
I_i(\mu) = \begin{cases}
- \sum_{\delta \in (\mu, 0) \cap \mathcal{D}_i} m_i(\delta) \hspace{2mm} \mbox{if } \mu <0 \\
\hspace{3mm} \sum_{\delta \in [0, \mu] \cap \mathcal{D}_i} m_i(\delta) \hspace{2mm} \mbox{if } \mu >0 \\
\end{cases}
\end{equation*}
where \(m_i(\delta)\) denotes the dimension of the corresponding vector space of \(\delta\)-homogeneous harmonic functions. Set \( \mathcal{F} = (\mathbb{R} \setminus \mathcal{D}_1 ) \times \ldots \times (\mathbb{R} \setminus \mathcal{D}_N) \subset \mathbb{R}^N \). If \( \mu \in \mathcal{F} \) then \( \Delta: C^{2, \alpha}_{\mu} (X) \to C^{\alpha}_{\mu-2} \) is Fredholm. Moreover, the index  is locally constant on \(\mathcal{F}\) and \( \mbox{ind}(\Delta) = - \sum_{i} I_i (\mu_i) \). The proof of these facts follows as in  the standard conifold case (see \cite{Joyce}, \cite{Behrndt}) once we are provided with Donaldson's interior Schauder estimates (Equation \eqref{DiSe}) together with the weighted estimate given by Proposition \ref{estimate cone}.

Let \(\mathcal{P}_i\) be the space of \(\delta\)-homogeneous plurisubharmonic functions on \(g_{F_i}\) with \( \delta \in [0, 2] \). Note that \(\mathcal{P}_i\) is made up of constant and possibly also linear functions in \(\mathbb{C}^2\). Let \(\mu_i>0\) such that \( (2, 2+ \mu_i] \cap \mathcal{D}_i = \emptyset\) for all \(i\).  Consider the vector space  
\begin{equation} \label{T0M}
T_0\mathcal{M}= \{ u \in C^{2, \alpha}_{\mu +2}(X) \oplus_i (1-\chi)( \mathcal{P}_i \oplus  \mathbb{R} \cdot r_i^2 \oplus \mathcal{H}_i) \circ \Phi_i^{-1} \hspace{2mm} \mbox{with} \hspace{2mm} \int_X u =0  \}
\end{equation}
where \(\mathcal{H}_i=0\) if \(d_i \geq 3\) and \(\mathcal{H}_i = \mathbb{R} \cdot h_i \) with \( h_i = |z_1|^{2\beta_{i_1}} - |z_2|^{2\beta_{i_2}} \) if \(d_i=2\). Similarly we let
\begin{equation} \label{T0N}
T_0\mathcal{N}= \{ f = \overline{f} + \sum_{i} f_{x_i} \hspace{2mm} \mbox{with} \hspace{2mm} \overline{f} \in C^{\alpha}_{\mu}(X), f_{x_i} \in \mathbb{R},  \int_X f =0  \}.
\end{equation}
These spaces are finite dimensional extensions of the subspaces of functions in \(C^{2, \alpha}_{\mu}\) and \(C^{\alpha}_{\mu}\) respectively with zero average. We fix any norm on the finite dimensional factors so that they become Banach spaces.   
\begin{proposition} \label{linearthm}
	\(\Delta_{\omega} : T_0 \mathcal{M} \to T_0 \mathcal{N} \) is an isomorphism.
\end{proposition}
This result is a straightforward consequence of the previous description of the co-kernel of \(\Delta_{\omega}\) in terms of homogeneous harmonic functions in a neighbourhood of the multiple points together with Lemma \ref{hom harm lemma} and the discussion after it. See third bullet in page 95 of \cite{heinsun}. 

\section{Yau's continuity path} \label{yau cont path sect}

\subsection{Initial metric}
For each multiple point \(x_i\) we let \(G_i\) the connected component of the transverse automorphism of the PK cone \(g_{F_i}\), that is the holomorphic automorphisms of \(\mathbb{C}^2\) that commute with the \(1\)-parameter group generated by the Reeb vector field and preserve the conical set. More explicitly the automorphisms in \(G_i\) are \(P_{i} (z_1, z_2) = (\lambda_{i, 1}z_1 , \lambda_{i, 2}z_2 ) \) with \(\lambda = (\lambda_{i, 1}, \lambda_{i, 2}) \) for some \(\lambda_{i, 1}, \lambda_{i, 2} >0 \) and \(\lambda_{i, 1} = \lambda_{i, 2}\) if \(d_i \geq 3\). We set \(\mathcal{M}\) to be the space of functions \(u\) of the form 
\begin{equation} \label{function in M}
  u = \overline{u} + (1-\chi)\sum_{i} \left( p_i + \frac{1}{2}(r_i^2 \circ P_i - r_i^2)  \right) \circ \Phi_i^{-1} 
\end{equation}
with \(\overline{u} \in C^{2, \alpha}_{\mu +2} \), \(P_i \in G_i \), \(p_i \in \mathcal{P}_i\), \( \omega + i \dd u >0 \) and \( \int_X u =0 \). We also let
\begin{equation*}
\mathcal{N}= \{ f = \overline{f} + \sum_{i} f_{x_i} \hspace{2mm} \mbox{with} \hspace{2mm} f \in C^{\alpha}_{\mu}(X), f_{x_i} \in \mathbb{R},  \int_X (e^f-1) =0  \} .
\end{equation*}
The Monge-Amp\`ere operator is
\begin{equation*}
M(u) = \log \frac{(\omega + i \dd u)^2}{\omega^2} .
\end{equation*}
The facts are that \( \mathcal{M} \) and \(\mathcal{N}\) are Banach manifolds and \( M: \mathcal{M} \to \mathcal{N} \) is a \(C^1\) map. The tangent spaces at the origin are given by equations \eqref{T0M} and \eqref{T0N} respectively. We have \(M(0)=0\) and \( \delta M|_0 = \Delta_{\omega} \).

Write \(\omega^2 = e^{-f}\Omega \wedge \overline{\Omega} \). We have normalized so that \(\int_X \omega^2 = \int_X \Omega \wedge \overline{\Omega} \), hence \(\int_X (e^f -1)\omega^2 =0 \). It follows from Lemma \ref{ricci lemma} that the Ricci potential \(f = \overline{f} + \sum_i f_{x_i} \) of the reference metric \(\omega=\omega_{ref} \) belongs to \(\mathcal{N}\), indeed in a neighbourhood of \(x_i\) we have that \(\overline{f} = h_i\) is a pluri-harmonic function vanishing at the origin so we can take any \(0< \mu <1\). Decreasing the H\"older exponent and the weight in \(C^{\alpha}_{\mu}\) we might assume that \(\overline{f}\) can be approximated to arbitrary order by \(h \in C^{\infty}_c (X') \), that is for any \(\epsilon>0\) there is a function \(h\) smooth in complex coordinates vanishing in a neighbourhood of the multiple points such that \( \|\overline{f} - h \|_{C^{\alpha}_{\mu}(X)} < \epsilon \). 

Let \(f_0 = h + \sum_i f_{x_i} \). We take \(\epsilon >0\) small and use Proposition \ref{linearthm} together with the implicit function theorem to solve for \( (\omega + i \dd \phi_0)^2 = c e^{f-f_0}\omega^2 \) with \(\phi_0 \in \mathcal{M} \) and the constant \(c>0\) is determined by integration over \(X\) and lies in \( e^{-\epsilon} \leq c \leq e^{\epsilon} \). 
\begin{lemma} \label{initial metric lemma}
The initial metric \( \omega_0 = \omega + i \dd \phi_0 \) has the following properties:
	\begin{itemize}
		\item \(\omega_0^2 =  c e^{-f_0} \Omega \wedge \overline{\Omega} \).
		\item For each \(i\) we have \( \| \Phi_{i, 0}^{*} \omega_0 - \omega_{F_i} \|_{\alpha, \mu} \leq C \) with \( \Phi_{i, 0} = \Phi_i \circ P_0 \) for some \(P_0 \in G_i\).
		\item \( C^{-1} \omega_{ref} \leq \omega_0 \leq C \omega_{ref} \).
	\end{itemize}
\end{lemma}
The first item follows by construction and the second and third items are an immediate consequence of the expression for \(\phi_0\) in Equation \eqref{function in M}. The function \(f_0\) is the Ricci potential of \(\omega_0\), that is \( \mbox{Ric}(\omega_0)= i \dd f_0 \). We already mentioned that \(f_0\) is smooth in complex coordinates and is constant \(f_0 \equiv f_{x_i} \) in a neighbourhood of each multiple point, in particular this implies that \( \|\mbox{Ric}(\omega_0)\|_{\omega_0} \leq C \). The property of \(\omega_0\) having uniformly bounded Ricci curvature is its main advantage over the reference metric \(\omega_{ref}\).

\subsection{The a priori estimates}
We set \( \omega_0\) to be the initial metric in Lemma \ref{initial metric lemma}. We change \(f_0\) to \(f_0 -\log c \) so that \(\omega_0^2 = e^{-f_0} \Omega \wedge \overline{\Omega}\). Our goal is to show that there is \(u \in \mathcal{M} \) such that \(\omega_u^2 = e^{f_0} \omega_0^2  \) where \(\omega_u = \omega_0 + i \dd u \). We use Yau's continuity path and introduce the set 
\begin{equation} \label{cont path}
	T = \{ t \in [0, 1] \hspace{1mm} \mbox{:} \hspace{1mm} \mbox{there is } u_t \in \mathcal{M} \hspace{2mm} \mbox{with} \hspace{2mm} \omega_{u_t}^2 = c_t e^{tf_0} \omega_0^2  \} .
\end{equation}
We use the simplified notation \(\omega_t = \omega_{u_t} \).

Write \(u_t \in \mathcal{M}\) as in Equation \eqref{function in M}
\begin{equation} 
u_t = \overline{u}_t + (1-\chi)\sum_{i} \left( p_i(t) + \frac{1}{2}(r_i^2 \circ P_i(t) - r_i^2)  \right) \circ \Phi_i^{-1} 
\end{equation}
with \(\overline{u}_t \in C^{2, \alpha}_{\mu +2} \), \(P_i(t) \in G_i \) and \(p_i(t) \in \mathcal{P}_i\). We have 
\(P_i(t) (z_1, z_2) = (\lambda_{i, 1}(t)z_1 , \lambda_{i, 2}(t)z_2 ) \) for some \(\lambda_{i, 1} (t), \lambda_{i, 2}(t) >0 \). The main result of this section is the following:

\begin{proposition} \label{apriori}
	There is \(C >0 \) independent of \(t \in T\) such that \( \|\overline{u}_t\|_{2, \alpha, 2+ \mu} \leq C \), \( C^{-1} \leq \lambda_{i, 1} (t), \lambda_{i, 2}(t) \leq C \) and \( \|p_i(t) \| \leq C \) where the last estimate is with respect to any fixed norm in the finite dimensional vector space \(\mathcal{P}_i\).
\end{proposition}

\begin{proof}
	There are five steps, the first three are the \(C^0\), \(C^2\) and \(C^{2, \alpha}\) estimates which together give us the point-wise bound \( \|u_t\|_{2, \alpha} \leq C \). The last two steps are the weighted estimates \( \|\overline{u}_t\|_{0, 2+ \mu'} \leq C \) and finally \( \|\overline{u}_t\|_{2, \alpha, \mu+2} \). The details go as follows:
	\begin{itemize}
		\item \( \| u_t \|_0 \leq C \). This is a straightforward extension of the standard case of  smooth metrics on closed manifolds, the technique is Moser iteration. Take \(p>1\), multiply the Equation \(\omega_t^2 = c_t e^{tf_0}\omega_0^2 \) by \(u_t |u_t|^{p-2}\) and integrate by parts to obtain
		
		\begin{equation}\label{int C0}
		\int_{X} |\nabla u_t|^2 \omega_0^2 \leq p \int_{X} u_t |u_t|^{p-2} (1-c_te^{tf_0}) \omega_0^2 .
		\end{equation}
		The integration by parts is valid due to our mild singularities, see \cite{brendle}. The Sobolev inequality for the metric \(\omega_0\) (see Equation \eqref{Sob ineq}) together with Equation \eqref{int C0} give us the bound \(\|u_t\|_{L^2} \leq C \) together with \(\|u_t\|^p_{L^{2p}} \leq C p \| u_t \|^{p-1}_{L^{p-1}}\). The estimate  \(\|u_t\|_0 = \lim_{p\to \infty} \|u_t\|_{L^p} \) follows by induction.
		
		\item \(C^{-1} \omega_0  \leq \omega_t \leq C \omega_0 \). The reference metric \(\omega_{ref}\) is quasi-isometric to the initial metric \(\omega_0 = \omega_{ref} + i \dd \phi_0 \) and we will derive the equivalent estimate  $C^{-1} \leq \mbox{tr}_{\omega_t} (\omega_{ref}) \leq C$.
		Note that \(\omega_t = \omega_{ref}  + i \dd \tilde{u}_t \) with \(\tilde{u}_t = u_t + \phi_0 \) and the previous bullet gives us a bound on \(\|\tilde{u}_t\|_0\). Standard elliptic regularity theory implies that \(u_t\) is smooth on the complement of the conical set. We want to derive a uniform bound on the function $ H = \log \mbox{tr}_{\omega_t} (\omega_{ref}) - A \tilde{u}_t$ with \(A>0\) a uniform constant independent of \(t\).

		Recall that \(\mbox{Bisec}(\omega_{ref}) \leq C_1 \). On the other hand the equation in our continuity path implies that
		\begin{equation}
		\mbox{Ric}(\omega_t) = (1-t) \mbox{Ric}(\omega_0) .
		\end{equation}
		Since the initial metric \(\omega_0\) has uniformly bounded Ricci curvature we easily deduce that \(\mbox{Ric}(\omega_t) \geq - C_2 \omega_{ref} \). 
		The Chern-Lu inequality tells us that: 
		\begin{equation} \label{chern lu ineq}
		\triangle_{\omega_t} \log \mbox{tr}_{\omega_t} (\omega_{ref}) \geq - (C_1 + C_2) \mbox{tr}_{\omega_t} (\omega_{ref}) .
		\end{equation}
		
		We take \(A= C_1 + C_2 +1\). We want to show  that $H = \mbox{tr}_{\omega_t} (\omega_{ref}) - A \tilde{u}_t$ is uniformly bounded above. Since $H(y) \to \log 2$ as $y \to x_i$, we can assume that $H$ attains its global maximum away from the multiple points. Moreover, by Jeffres' barrier trick\footnote{I.e. replacing \(H\) by \(H + \sum_{j=1}^{n} |s_j|^{2\epsilon}_h\) with \(\epsilon>0\) sufficiently small
		and where \(h\) is a smooth Hermitian metric on \(\mathcal{O}(1)\) and \(s_j\) are defining sections for the complex lines \(L_j\).} (see \cite{Jeffres}) we can assume that the maximum of \(H\) is not attained in the conical set. At the point of maximum \(x\) we have 
		\begin{equation*}
			0 \geq \triangle_{\omega_t} H (x) \geq -Q \mbox{tr}_{\omega_t} (\omega_{ref}) - A\triangle_{\omega_t} \tilde{u}_t = \mbox{tr}_{\omega_t} (\omega_{ref}) (x) - 2A ,
		\end{equation*}
		where we have used Equation \eqref{chern lu ineq} together with the identity $2 = \mbox{tr}_{\omega_t} (\omega_{ref}) + \triangle_{\omega_t} \tilde{u}_t$.
		It is then easy to conclude that $\mbox{tr}_{\omega_t} (\omega_{ref}) \leq C$. 
		We also have \( \omega_t^2 = c_t e^{(t-1)f_0 +f} \omega_{ref}^2 \), therefore the determinants of the metrics are uniformly equivalent and hence we must also have $\mbox{tr}_{\omega_t} (\omega_{ref}) \geq C^{-1}$. 
		
		It is a direct consequence of \(C^{-1} \omega_0  \leq \omega_t \leq C \omega_0 \) that \( C^{-1} \leq \lambda_{i, 1} (t), \lambda_{i, 2}(t) \leq C \). On the other hand the \(C^0\) together with the \(C^2\)-estimate give us a uniform bound on the gradient of \(u_t\). Since the \(\mathcal{P}_i\) are made up of constant and linear functions and \(\overline{u}_t\) together with its gradient vanish at the multiple points, we conclude that \(\|p_i(t)\|\leq C\).
		
		\item \(\|u_t\|_{C^{2, \alpha}_{loc} (X')} \leq C\). More precisely, for any compact \(K \subset X' \) we have a uniform bound \([\dd u_t]_{\alpha, K} \leq C \) and for points \(p\) in a neighbourhood of \(x_i\) we have \([\dd u_t]_{\alpha, B(p, r_i(p)/2)} \leq C \). These estimates are an immediate consequence of the interior Schauder estimates for the Monge-Amp\`ere operator, see \cite{ChenWang}.
		
		\emph{Note:}
		Before proceeding with the weighted estimates  we remark that the automorphisms of the PK cones preserve the weighted function spaces we have introduced and, since the automorphisms remain bounded, the norms defined with respect to the identifications \( \Phi_i \circ P_i(t)^{-1} \) are uniformly equivalent. We also observe that if we replace \(u_t\) by \(\overline{u}_t\) we have \(\omega_{\overline{u}_t}^2 = e^{\overline{f}_t} \omega_0^2 \) with \(\overline{f}_t\) uniformly bounded.
		
		\item \( \|\overline{u}_t\|_{2+\mu', 0} \leq C \) where \(\mu'\) is any fixed tuple of positive numbers with \(0<\mu'<\mu\). This estimate follows from a weighted Moser iteration, the argument goes back to theorem 8.6.6 in \cite{Joyce}. We introduce a weight function \( 0 < \rho \leq 1\) on \(X'\) with \(\rho \equiv 1\) away from the multiple points and \(\rho= r_i \) in a neighbourhood of \(x_i\). We introduce the norm
		$$ \| \overline{u} \|_{L^p_{\delta}}^p = \int_{X} |\overline{u}|^p \rho^{-p\delta} \rho^{-4} \omega_0^2 . $$ 
		We let \(\delta = 2 + \mu' \) so that \(\overline{u}_t \in L^p_{\delta}\) for any \(p \geq 1\). The Sobolev inequality together with an integration by parts argument lead to \(\| \overline{u}_t\|^p_{L^{2p}_{\delta}} \leq Cp \left( \| \overline{u}_t\|^{p-1}_{L^{p-1}_{\delta}} + \|\overline{u}_t\|^p_{L^p_{\delta}} \right) \), as noticed by Joyce \cite{Joyce} the uniform bound \(\| \dd \overline{u}_t \|_{\omega_0} \leq C\) obtained before is necessary to derive the inequality. The estimate \( \|\overline{u}_t\|_{\delta, 0} = \lim_{p \to \infty} \|\overline{u}_t\|_{L^p_{\delta}} \leq C \) follows by induction as in the \(C^0\) case once we estimate \(\|\overline{u}\|_{L^p_{\delta}}\) for some \(p\). In order to obtain the initial integral estimate we note that the Sobolev inequality \eqref{Sob ineq2} give us a uniform bound \( \int_X \rho^{2 \alpha -4} |\overline{u}_t|^{\alpha} \leq C \) for any \( 1 < \alpha < 2\). We use the H\"older inequality \( ab \leq a^r/r + b^s/s \) with \( a = |\overline{u}_t|^p \rho^{-p\delta -1} \) and \( b = \rho^{-3} \), \(r = (4-2\alpha)/(1+ p \delta)^{-1} \), \(r^{-1} + s^{-1} =1 \). We take \( 1 < \alpha = pr < 3/2\) and conclude that \(\int_X |\overline{u}_t|^{p} \rho^{-p\delta-4} \leq C \).

		\item \( \|\overline{u}_t\|_{2, \alpha, \mu +2} \leq C \). Given the previous  \( \|\overline{u}_t\|_{\mu'+2, 0} \leq C \) bound, this is a non-linear analogue of Lemma \ref{Est 1}. The argument goes back to theorem 8.6.11 in \cite{Joyce}.
		We write the equation $\omega_{\overline{u}_t}^2 = e^{\overline{f}_t} \omega_0^2$  as 
		\begin{equation} \label{EQ2}
		\triangle_{\overline{u}/2} \overline{u} = H (e^{\overline{f}_t}-1), 
		\end{equation}
		where $\triangle_{\overline{u}/2}$ is the Laplace operator of $\omega_{\overline{u}/2} = \omega_0 + i \partial \overline{\partial} (\overline{u}/2)$ and $H = \omega_0^2 / \omega_{\overline{u}/2}^2$.
		Note that $\omega_{\overline{u}/2} = (1/2) \omega_0 + (1/2) \omega_{\overline{u}} \geq (1/2) \omega_0$. We have a bound on the $C^{2, \alpha}_{loc}(X')$ norm of  $\overline{u}$ which give us a \(C^{\alpha}\) bound on the coefficients of \(\triangle_{\overline{u}/2}\).  The interior Schauder estimates give us
		\begin{equation} \label{int1}
		\| \overline{u} \|_{C^{2, \alpha} (B(p, r(p)/2)))} \leq C \left( \| \triangle_{\overline{u}/2} \overline{u} \|_{C^{2, \alpha} (B(p, r(p))} + \| \overline{u} \|_{C^0(B(p, r(p)))} \right) , 
		\end{equation}
		with a constant $C$ independent of $p$. We multiply \eqref{int1} by $\rho(p)^{-\mu'}$ and use the previous \(C^0_{\delta}\)-estimate to obtain the bound \( \| \dd \overline{u}_t \|_{\alpha, \mu'} \leq C \).
		We improve the weight on the estimate by writing the equation $\omega_{\overline{u}_t}^2 = e^{\overline{f}_t} \omega_0^2$  as 
		\begin{equation} \label{EQ1}
		\triangle_{\omega_0} \overline{u}_t = (e^{\overline{f}_t} -1) + \psi , 
		\end{equation}
		with $\psi = (\dd \overline{u}_t \wedge \dd \overline{u}_t )/ \omega_0^2 $ and appealing to Proposition \ref{estimate cone}. 
		
	\end{itemize}
	 
\end{proof}

\emph{Proof of Theorem \ref{thm1}.} The set \(T\) is non-empty with \(u_0 \equiv 0\), it is open thanks to Proposition \ref{linearthm} and it is closed because of Proposition \ref{apriori}. We conclude that \(T = [0, 1] \), in particular \(1 \in T\) and we set \(\omega_{RF} = \omega_1\) with \(\Psi_i = \Phi_i \circ P_i (1)^{-1} \). 

The rate \(\mu\) at \(x_i\) can be taken to be any positive number \(0 < \mu < \min \{\lambda-2, 1\} \) where \(\lambda\) is the first growth rate bigger than two of a non-zero homogeneous harmonic function on the PK cone \(g_{F_i}\). The restriction that \(\mu < 1\) comes from the fact that the Ricci potential of the reference metric can be written as \(f = \overline{f} + \sum_i \overline{f}_{x_i} \) with \(\overline{f}_{x_i}\) constant and \(\overline{f} \in C^{\alpha}_{\delta} \) for some \(\delta \geq 1\).
\qed

\begin{remark}
	If a line arrangement \(\{L_j\}_{j=1}^n\) is invariant under a non-trivial \(\mathbb{C}^{*}\)-action then it is given by a collection of lines in \(\mathbb{C}^2\) going through the origin and possibly also the line at infinity. In any case there are at least \(n-1\) lines going through a point and there are no possible values of \(0<\beta_j<1\) that make \((\mathbb{CP}^2, \sum_{j=1}^{n}(1-\beta_j)L_j)\) a klt log CY pair, that is equations \eqref{cond1} and \eqref{cond3} are satisfied. We conclude that none of the arrangements considered in Theorem \ref{thm1} is invariant under a non-trivial \(\mathbb{C}^{*}\)-action. Following \cite{donaldson1}, we can appeal to Theorem \ref{linearthm} together with the implicit function theorem to show existence of K\"ahler-Einstein metrics for small perturbations of the cone angles. This matches with \cite{fujita}, since small perturbations which are log Fano (\(\sum_{j=1}^{n}(1-\beta_j)<3\)) are necessary log \(K\)-stable according to theorem 1.5 in \cite{fujita}.  Moreover, let us note that this also gives the first examples of singular (non-orbifold or log-smooth) KE metrics which are \emph{non-Ricci flat} for which a polynomial decay of the metric at the tangent cones is obtained.
\end{remark}

\section{Chern-Weil integrals} \label{chernweilsect}

\subsection{Energy formula}
Recall that the energy of a Riemannian manifold \((M, g)\) is defined as
\begin{equation*}
E(g) = \frac{1}{8\pi^2} \int_M | \mbox{Riem}(g)|^2 dV_g .
\end{equation*}
This is a scale invariant quantity in four real dimensions. If \(g\) is a smooth Ricci-flat metric on a closed \(4\)-manifold then its energy is equal to the Euler characteristic of the manifold, that is \( E(g) = \chi(M)\). In this section we discuss a Chern-Weil type formula for the energy of the metric \(g_{RF}\) in Theorem \ref{thm1}.

It follows from the polyhomogeneous expansion that the energy distribution \(| \mbox{Riem}(g_{RF})|^2 \) is locally integrable at points of \(L_j^{\times}\). Indeed the polyhomogeneous expansion implies that \( |\mbox{Riem}(g_{RF})| = O(\rho^{1/\beta_j -2}) \) and the local integrability follows by comparison with \(\int_{0}^{1} \rho^{2/\beta_j -3} d\rho < + \infty \). At multiple points of the arrangement we have the next.

\begin{conjecture}\label{conj:E}
	In the setting of Theorem \ref{thm1} we expect that the following assertions hold.
	\begin{enumerate}
		\item The energy distribution \(|\mbox{Riem}(g_{RF})|^2\) is locally integrable at the multiple points of the arrangement.
		\item The energy of \(g_{RF}\) is given by:
		\begin{equation} \label{energy formula}
		E(g_{RF}) = 3 + \sum_{j=1}^{n} (\beta_j -1) \chi (L_j^{\times}) + \sum_{i}(\nu_i -1),
		\end{equation}
		where $\nu_i$ denotes the volume density of the Ricci-Flat metric at the multiple points $x_i$.
	\end{enumerate}
\end{conjecture}

The right hand side of Equation \eqref{energy formula} can be interpreted as a `logarithmic' Euler characteristic where points are weighted by the volume density of the metric, similarly to the Gauss-Bonnet formula for conical metrics on Riemann surfaces
\[\frac{1}{2\pi} \int_X K_g dV_g = \chi(X) + \sum_i (\beta_i-1) \] 
where \(X\) is a compact surface endowed with a metric \(g\) with cone angles \(2\pi\beta_i\) at a finite number of points \(x_i\) and \(K_g, dV_g\) denote its Gaussian curvature and area form.
In the context of Theorem \ref{thm1} the volume density at a point \(p\) is equal to \(1\) if \( p \in \mathbb{CP}^2 \setminus \cup_j L_j \), to \(\beta_j\) if \(p \in L_j^{\times} \) and to \(\nu_i\) if \(p = x_i \). Since the Euler characteristic is additive, or `motivic', under disjoint union, we can rewrite Equation \eqref{energy formula} as \(	E(g_{RF}) = \chi (\mathbb{CP}^2 \setminus \cup_j L_j ) + \sum_{j=1}^{n} \beta_j \chi(L_j^{\times}) + \sum_i \nu_i \chi (\{x_i\})\).

\subsection*{Explanation of Equation \eqref{energy formula}}
Let \( \nabla \) be the Levi-Civita connection of \(g_{RF}\) and \(F_{\nabla}\) be its curvature. The metric is anti self-dual and we have the point-wise identity  \( |\mbox{Riem}(g)|^2 dV_g=- \mbox{tr}(F_{\nabla} \wedge F_{\nabla}) \). The energy is given by the Chern-Weil integral
\begin{equation*}
E = -\frac{1}{8\pi^2} \int_{\mathbb{CP}^2} \mbox{tr}(F_{\nabla} \wedge F_{\nabla}) .
\end{equation*}

We remove small balls around the multiple points and let \( Z = \mathbb{CP}^2 \setminus \cup_i B(x_i, \epsilon) \). Under suitable assumptions, it is reasonable to expect that the following formula holds
\begin{equation}\label{Econj1} 
E (g|_{Z}) = \chi(Z) + \sum_{j=1}^{n} (\beta_j -1) \chi (L_j^{\times}) + \sum_i \frac{1}{2\pi^2} \int_{\partial B(x_i, \epsilon) } \ \left( det (II) + \langle II, \hat{\mathcal{R}} \rangle \right),
\end{equation}
where \(II\) denotes the second fundamental form of \(\partial B(x_i, \epsilon) \subset Z \) and $\hat{\mathcal{R}}$ is the restriction of the ambient curvature operator thought as a symmetric two tensor by means of the three dimensional Hodge star operator. The above formula for \(E (g|_{Z})\) results from a mix of the case of smooth metrics on compact manifolds with boundary (see equation 5.16 in \cite{AL}) with the case of metrics with conical singularities along a smooth complex curve in a closed surface (see \cite{SW}).

The boundary contribution corresponds to the Chern-Simons integral  
$$ CS (\nabla, Y) = \frac{1}{2 \pi^2}\int_{Y} dA \wedge A + \frac{2}{3}A\wedge A \wedge A$$
with \( \nabla = d + A \)  and \( Y = \partial B(x_i, \epsilon) \). On the other hand, it is easy to see that for a \(PK\) cone \(g_F= dr^2 + r^2 \overline{g} \) we have
\begin{equation*} 
CS (\nabla_{g_F}, \partial B(0, r)) = \mbox{Vol}(\overline{g})/2\pi^2 = \nu
\end{equation*} 
for any \( r>0\). Indeed $\hat{\mathcal{R}}_{g_F} \equiv 0$ and \(\det(II_{g_F}) \equiv 1/r^3\).

Equation \eqref{energy formula} follows if we show that  we can replace the boundary integral of \(g\) with that of \(g_F\) in the limit as \( \epsilon \to 0 \). Our proof of Theorem \ref{thm1} give us control on the decay of the K\"ahler potential \(\phi\) of \(g_{RF}\) up to second order derivatives \(\dd \phi \). At this point we require further control on derivatives of \(\phi\) up to fourth order.
Special care must be taken with derivatives in transverse directions to the conical set. 
We \emph{conjecture} that it is possible to upgrade the asymptotic given by Theorem \ref{thm1} to the following: 
\begin{enumerate}
	\item[(i)] Taking derivative with respect to the vector field \(\partial_r\) (which is tangent to the conical set) we can upgrade Equation \eqref{asymptotics} to 
	\begin{equation*} 
	| \nabla_{\partial_r} (g - g_F) |_{g_F} = O (r^{\mu-1}) .
	\end{equation*}
	We get that \( \int_{\partial B(0, \epsilon)} | \det (II_g) - \det (II_{g_F}) | = O (\epsilon^{\mu}) \).
	\item[(ii)] We expect that \( \int_{\partial B(0, \epsilon)}
	\langle II_g, \hat{\mathcal{R}_g} \rangle = O(\epsilon^{\mu}) \). The idea is that the curvature only involves mixed complex derivatives that one could in principle estimate via Donaldson's interior Schauder estimates together a bootstrapping argument.
\end{enumerate}
If the previous two bullets are established then we can evaluate the terms \(\mbox{det}(II)\) and \(\langle II, \hat{\mathcal{R}} \rangle\) in Equation \eqref{Econj1} as \(\epsilon \to 0\) by replacing them with the ones of the corresponding PK cone. To sum up, if items \((i)\) and \((ii)\) are proved then our Conjecture \ref{conj:E} follows.

\subsection*{Ordinary double points}
We consider the generic case, that is \(d_i =2\) for all \(i\). We have
\begin{equation*}
| \cup_i \{x_i\} | = \binom{n}{2} , \hspace{2mm} \chi (L_j^{\times}) = 2 - (n-1) .
\end{equation*}  
The second term in Equation \eqref{energy formula} is \( (3-n) \sum_{j=1}^{n} (\beta_j -1) = 3(n-3) \). On the other hand
\begin{equation*}
\sum_i \nu_i = \sum_i \beta_{i_1} \beta_{i_2} = \frac{1}{2} \sum_{j, k, j \neq k}^n \beta_j \beta_k .
\end{equation*} 
Since \( \sum_{k=1, k \neq j}^{n} \beta_k = n-3-\beta_j \), we get 
\begin{equation*}
\sum_{j, k, j \neq k}^n \beta_j \beta_k = (n-3)\sum_{j=1}^{n}\beta_j - \sum_{j=1}^{n} \beta_j^2 = (n-3)^2 - \sum_{j=1}^{n} \beta_j^2  . 
\end{equation*}
It is easy to check that \( \sum_{j=1}^{n}(1-\beta_j)^2 = (\sum_{j=1}^{n} \beta_j^2) -n +6 \). A simple computation  shows that
\begin{equation} \label{odp energy}
E = \frac{3}{2} - \frac{1}{2}\sum_{j=1}^{n}(1-\beta_j)^2 .
\end{equation}

This matches with Tian's formula (equation 2.6 in \cite{Tian}) 
\begin{equation} \label{Tian formula}
\begin{split}
c_2 (M, g) = c_2(M) + \sum_{j=1}^{n} (1-\beta_j) (K_M \cdot D_j + D_j^2) \\ 
+ \frac{1}{2} \sum_{\substack{j, k \\ j\neq k}}^{n} (1-\beta_j)(1-\beta_k) D_j \cdot D_k - \sum_{j=1}^{n}(1-\beta_j^2) \mbox{Sing}(D_j)
\end{split}
\end{equation}
In our case \(M = \mathbb{CP}^2 \) and \(D_j = L_j\). We identify \(H^4_{dR}(M) \cong \mathbb{R} \) via integration, \(\mbox{Sing}(D_j) = 0 \) since the singular locus of \(D_j\) is empty and the adjunction formula gives us  \( K_M \cdot D_j + D_j^2 = -\chi(D_j) = -2 \). It is then easy to check that \(c_2(M, g)\) given by Equation \eqref{Tian formula} is equal to \(E\) given by Equation \eqref{odp energy}.

Write \(t_j = 1- \beta_j\), so \(\sum_{j=1}^{n} t_j = 3 \) and \( 0 < t_j < 1 \). Since \(t_j^2 < t_j\) for each \(j\) we have \(2 E = 3 - \sum_{j=1}^{n}t_j^2 >0 \). The energy is a concave function of \(t_1, \ldots, t_n\) and attains its maximum when \(t_j = 3/n\) for all \(j\), its infimum is \(0\) and it's not attained. Note that the value \(\beta=\frac{n-3}{n}\) is the Calabi-Yau threshold for hypersurfaces of degree \(n\) in \(\mathbb{C P}^2\). If \(g_s\) is a sequence with \(E(g_s) \to 0 \) as \(s \to \infty \) the we can relabel so that the corresponding \(t_j \to 1\) for \(j=1, 2, 3\)  and \(t_j \to 0\) if \(j \geq 4\) as \(s \to \infty\), which corresponds to the sequence \(g_s\) converging (after rescaling to avoid collapsing) to \(\mathbb{C}^{*} \times \mathbb{C}^{*} \) with its flat metric \( (S^1 \times \mathbb{R})^2 \).

\subsection{Parabolic bundles} Let \(X\) be a complex surface and \( D = \cup_k D_k \subset X \) a finite collection of smooth irreducible complex curves with normal crossing intersections. Let \(E\) be a rank two holomorphic vector bundle on \(X\). A collection of holomorphic line sub-bundles \(L_k \subset E|_{D_k} \) together with weights \( 0 < \alpha_k < 1 \) endows \(E\) with the structure of a \emph{parabolic bundle}. Much of the standard theory on holomorphic vector bundles has been extended to this setting. In particular there are parabolic versions of: Chern classes, slope stability, Bogomolov-Gieseker inequality and Hitchin-Kobayashi correspondence. See \cite{panov} and references therein.

It is a standard fact that \(\omega\) is a smooth KE on \(X\)  if and only if \((TX, h)\) is a Hermitian-Einstein vector bundle over \((X, \omega)\) where \(h\) is the Hermitian metric defined by \(\omega\). In particular, the tangent bundle of a KE manifold is slope stable with respect to the polarization determined by the K\"ahler class. There is a natural extension to the case of a KEcs on \(X\) with cone angle \(2\pi\beta\) along a smooth divisor \(D \subset X \). The parabolic structure on \(TX\) is given by \(L = TD \subset TX|_D\) with weight \(\alpha=1-\beta\). The KEcs metric induces a Hermitian metric on \(TX\) compatible with the parabolic structure and it satisfies the Hermitian-Einstein equations. See section 6 in \cite{kellerzheng}.

In the more general setting of a weak KEcs metric on a klt pair we expect to be a natural parabolic structure (endowed with a compatible Hermitian-Einstein metric) on the pull-back of the tangent bundle to a log resolution of the pair. We restrict the discussion to our case at hand, namely  a weighted line arrangement such that \((\mathbb{CP}^2, \sum_{j=1}^{n}(1-\beta_j)L_j)\) is a log Calabi-Yau klt pair. We follow \cite{panov}:

Let \( \pi: X \to \mathbb{CP}^2 \) be the blow-up at the multiple points \(x_i\) of the arrangement with multiplicity \( d_i \geq 3\). Let \(D = \cup_{j=1}^{n+m}D_j \subset X \) be the union of the proper transforms of the lines \(D_1, \ldots, D_n\) together with the exceptional divisors \(D_{n+1}, \ldots, D_{n+m}\). We let \(E = \pi^{*} T \mathbb{CP}^2 \) and note that \(E|_{D_j} \cong TX|_{D_j} \) if \(j=1, \ldots, n\) and \(E|_{D_j} \cong \underline{\mathbb{C}}^2\) (the trivial rank two vector bundle) if \(j \geq n+1\). We define a parabolic structure on \(E\), denoted by \(E_{*}\), by letting \( L_j = TD_j \), \(\alpha_j = 1-\beta_j\) if \(j=1, \ldots, n\) and for \(j \geq n+1\) we set \(L_j = \{0\}\), \(\alpha_j = 1 - \gamma_{i} \) where \(\pi(D_j) = x_i \) and  \(\gamma_i\) is given by Equation \eqref{angle mult points}. The parabolic Chern numbers of \(E_{*}\) are computed in Proposition 7.1 in \cite{panov},  \(par ch_1 (E_{*}) = 0\) and \(parch_2(E_{*})\) is equal to the right hand side in Equation \eqref{energy formula}. It is proved in \cite{panov} that \(E_{*}\) is stable with respect to an appropriate polarization. The identity \(E(g_{RF}) \geq 0\) is equivalent to the Bogomolov-Gieseker inequality
\begin{equation*}
parch_2(E_{*}) - \frac{1}{2} parch_1^2(E_{*}) \geq 0 .
\end{equation*}

\section{General Picture (\(\dim_{\mathbb{C}} =2 \))} \label{conjectural section}

\subsection*{Weak KE metrics on klt pairs}
Let \(X\) be a projective surface, for simplicity we restrict to the case where \(X\) is smooth (rather than normal). Consider an  \(\mathbb{R}\)-divisor \(D=\sum_{j=1}^{n}(1-\beta_j)D_j\), where   \( D_j \subset X \) are distinct complex irreducible curves and \(0 < \beta_j < 1 \). Write \( C = \mbox{Supp}(D)= \cup_{j=1}^n D_j \) and let \( \mbox{Sing}(X, D) \) be the finite set of points at which the curve \(C\) is singular, that is the points at which either one of the irreducible curves is singular or where at least two intersect. In complex coordinates centred at \(p \in C\) we have defining equations \(D_j = \{ f_j =0 \} \) ) and \(C =  \{f=0\} \) with \( f= \prod_{j=1}^{n}f_j \) . The pair \((X, D)\) is said to be klt if for every \( p \in \mbox{Sing}(X, D) \) the function \( \prod_{j=1}^{n} |f_j|^{2\beta_j -2} \) is locally Lebesgue integrable around \( 0 \in \mathbb{C}^2 \). At non-singular points of \(C\) this local integrability condition is automatically satisfied since \( 0 < \beta_j <1 \). 

The first Chern class of the pair is defined as 
\begin{equation}
c_1 (X, D) = c_1(X) - \sum_{j=1}^{n} (1- \beta_j)c_1([D_j]) \in H^{1,1}(X, \mathbb{R}).
\end{equation}
A weak conical KE metric on \((X, D)\) is a smooth K\"ahler-Einstein metric with \( \mbox{Ric}(g) = \lambda g \) on the complement of \(C\)  such that in complex coordinates \((z, w)\) around \( p \in C\), \(g\) has a continuous K\"ahler potential and
\[ dV_g = F \prod_{j=1}^{n} |f_j|^{2\beta_j -2} dz dw \overline{dzdw} \] 
with \(F\) a positive continuous function. We consider the following cases:

\begin{enumerate}
	\item \(c_1(X, D) <0 .\) There is a unique weak conical KE metric with \( \lambda=-1 \). 
	\item \(c_1(X, D) = 0 .\) In each K\"ahler class there is a unique weak conical KE metric with \( \lambda=0 \).
	\item \(c_1(X, D) >0 \) and \((X, D)\) is log K-polystable. It is conjectured and partially proved (see \cite{SSY} for the smoothable setting for \( D\) plurianticanonical and, in more generality, the recent pre-print \cite{litianwang}) that there is a unique weak conical KE metric with \( \lambda=1 \).
\end{enumerate}

The weak KE metric induces a distance \(d\) on \(X \setminus C \) and it is natural to expect that the metric completion of \((X \setminus C, d)\) is homeomorphic to \(X\).  We want to understand the metric tangent cones of \(g\) at points in \(C\). Without any further assumptions, if \( p \in D_j \setminus \mbox{Sing}(X, D) \) then the regularity theory for conical metrics implies that \( T_p (X, g) =  \mathbb{C}_{\beta_j} \times \mathbb{C}  \). At present, there is no result that always guarantees the existence of a tangent cone at points in \(\mbox{Sing}(X, D)\). If  we assume \(g\) is a non-collapsed limit of a sequence of KE metrics with cone angle \(2 \pi \beta \) along smooth curves \( C_{\epsilon} \subset X \) with \(C_{\epsilon} \to C = \cup_{j=1}^n D_j \) as \(\epsilon \to 0 \), then we can appeal to \cite{CDS1} and \cite{CDS} to ensure the existence of tangent cones; but this imposes strong restrictions on the angles \(\beta_1, \ldots, \beta_n \) (if \(C_{\epsilon}\) converges to \(D_j\) with multiplicity \(k\) then \( 1- \beta_j = k(1-\beta)\)). 
More generally, one could consider the case when \(g\) is the limit of KEcs on log smooth pairs and appeal to results announced in \cite{litianwang}.

\subsection*{Dependence of \(C(Y)\) on the singularity and the value of the cone angle}

In analogy to \cite{DSII} we expect that the tangent cone \( T_p (X,g) \) is uniquely determined by the curve singularity at \(p\) and the cone angle parameters. In higher dimensions we expect  \(T_p (X, g) = C(Y) \) to be a metric cone over a `conical Sasaki-Einstein' manifold \((Y, \overline{g})\). In our case, it is possible to argue that \(C(Y)\) must be a PK cone. 

From an algebro-geometric point of view, by the recent theory developed in \cite{Li15}, the tangent cone \( T_p (X,g) \) could be in principle understood by looking at the \emph{infimum of normalized volume of valuations} centered at the point $p \in (X,D)$.

$$\hat{\mbox{vol}}((X,D),p):= \inf_{\nu\in Val_p}A_D^n(\nu)\, \mbox{vol}(\nu)>0, $$
where $A_D(\nu)$ is the log-discrepancy of a valuation and $$\mbox{vol}(\nu):=\limsup_{r\rightarrow 0}\frac{ \mbox{length}(\mathcal{O}_{V,p}/\{f|\nu(f)\geq r\})}{r^n/n!}$$ its volume.  The infimum is actually a minimum \cite{Blum}. This new invariant of klt singularities depends only on the local germ of the pair at $p$, and it correspond to the value of the local density $\nu$ of the metric at $p$ up to the multiplicative constant $n^n$ \cite{heinsun, LiXu}.  

Jumping phenomena of the tangent cone, i.e.,  when the tangent cone is not locally biholomorphic to a neighbourhood of the singularity, can be seen to happen already in our two dimensional log case. 

In the following conjectural examples we fix \(f \in \mathcal{O}_{\mathbb{C}^2, 0} \) and study the natural different \(C(Y)\)s cone associated to such curve as we vary the angle $\beta$. Naively, our candidate for \(C(Y)\) is the PK cone whose conical set `best approximates' \(\{f=0\}\). We write \( f = P_d + \mbox{higher order terms} \),  with \(P_d\) homogeneous of degree \(d = \mbox{ord}_0(f) \geq 2 \).

\subsubsection*{Case \(f\) is irreducible in \(\mathcal{O}_{\mathbb{C}^2, 0}\)} We change coordinates so that \(P_d = w^d\). We have a Puiseux series \( w = a_e t^e + \ldots \) with \( t = z^d \) and \( d < e\). The curve \(C = \{f=0\}\) is `approximated by' \( C' = \{ w^d = a z^e \}\) for some \(a \in \mathbb{C}\). The klt condition on the pair \((\mathbb{C}^2, (1-\beta)C)\) requires that \(\beta > 1 - 1/d - 1/e \). Indeed \( \beta > 1 - \mbox{l.c.t.}(f, 0)\), where 
\[ \mbox{l.c.t.}(f, 0) = \sup \{ c > 0 \hspace{1mm} \mbox{s.t.} \hspace{1mm} |f|^{-2c} \in L^1_{loc}  \}  \]
is the \emph{log canonical threshold} given by 
\begin{equation}\label{eq:lct}
\mbox{l.c.t.} (f, 0) = \frac{1}{d} + \frac{1}{e} , 	
\end{equation}
see \cite{kollar}, indeed Equation \eqref{eq:lct} goes back to Igusa \cite{Igusa}.

Let us try to guess the natural choice of the best approximating $C(Y)$ cones from differential geometry.
When \( 1 - 1/d - 1/e < \beta < 1 - 1/d + 1/e=: \beta^{*} \) there is a quasi-regular PK cone metric \(g_F\) with cone angle \( 2 \pi \beta \) along \(C'\) and we set \(C(Y)=g_F\). It corresponds to a spherical metric \(g_{1/e, 1/d, \beta}\) on the Riemann sphere with three conical singularities of angle \( 2 \pi (1/e), 2 \pi (1/d) \) and \(2 \pi \beta \). More precisely, \(g_{1/e, 1/d, \beta}\) lifts to a regular PK cone in \(\mathbb{C}^2\), \(\tilde{g}_F\), with cone angles \( 2 \pi (1/e) \) along \(\{u=0\} \), \( 2 \pi (1/d) \) along \(\{v=0\} \) and \( 2 \pi \beta \) along \(\{v=a u\} \). Let \( (u, v) = \Phi (z, w) = (z^e, w^d) \), then \( g_F = \Phi^{*} \tilde{g_F} \).  

It follows from elementary spherical geometry that 
\[ \lim_{\beta \to \beta^{*}} g_{1/e, 1/d, \beta} = g_{1/e, 1/e}\]
in the Gromov-Hausdorff sense.
Indeed, the spherical triangle with angles \((\pi/e, \pi/d, \pi \beta)\) degenerates as \(\beta \to \beta^*\) (in the Hausdorff sense) to a spherical bigon with angles \((\pi/e, \pi/e)\), see Section 3.1.3 in \cite{MP}. The spherical metric \(g_{1/e, 1/e}\) lifts through the Hopf map to the regular PK cone \( \tilde{g_F} = \mathbb{C}_{1/e} \times \mathbb{C}_{1/e} \) and \( \Phi^{*} \tilde{g_F} = \mathbb{C} \times \mathbb{C}_{d/e} \). Let \( \gamma^{*} = d/e \), note that \( 1 - \gamma^{*} = d (1 - \beta^{*}) \) which agrees with the picture of the \(d\) conical branches of \(C\) coming together.

Therefore, for \( \beta^{*} \leq \beta < 1 \),  we naturally set $$ C(Y) = \mathbb{C} \times \mathbb{C}_{\gamma^{*}} $$ with \( 1 - \gamma^{*} = d (1 - \beta) \). Thus \(\beta=\beta^{*}\) is the value of the parameter for which see the jumping of the best approximating cone $C(Y)$.

Note that the volume density of such Calabi-Yau cones when \(\beta\) varies is given by:

\begin{equation} \label{vol dens irred f}
\nu(f, \beta) = \begin{cases}
e  d \gamma^2 \quad \hspace{5mm} \mbox{if } 1- \mbox{l.c.t}(f, 0) < \beta < \beta^{*} \\
\gamma^{*} \quad \hspace{9.5mm} \mbox{if } \beta^{*} \leq \beta <1 \\
\end{cases}
\end{equation}
where \(2\gamma = 2 + (1/e-1) + (1/d-1) + (\beta -1) \) and \( \gamma^{*} = d \beta + 1 -d\). 

Now we try to make a similar computation using valuations. Similarly to \cite{Li15} section \(5\), we can consider monomial valuations induced by the \((x,y)\)-weighted scaling action on the ambient \(\mathbb{C}^2\). We have \(\mbox{vol}(\nu_{x,y})=1/xy\), and the log-discrepancy is equal to \(A_D(\nu_{x,y})= x+y-(1-\beta)\min\{dx,ey\}\). Thus the computation which should give the volume density  corresponds to \emph{minimize} the function
\begin{equation}
F(x,y)=\frac{(x+y-(1-\beta)\min\{dx,ey\})^2}{4xy}
\end{equation}
for \(x,y>0\). It is now straightforward to check that, when \(\beta<\beta^*\), the minimum is achieved at \(y=\frac{d}{e}x\) and the value of \(F\) at that point is equal to \(\frac{(e+d+ed(\beta-1))^2}{4ed}\), while for \(\beta\geq\beta^*\)  the minimum is achieved at \(y=(d\beta+1-d)x\)  and the value of \(f\) at that point is equal to \(d\beta+1-d\). Thus the minimum of \(f\) coincides with the volume density computed before, as it should be.

\begin{remark}
	Assuming the existence of tangent cones, the above algebraic computation, thanks to the results of \cite{LiXu}, can possibly be used to see that the tangent cones agrees with the one we described. However, even in such case and under non-jumping assumptions, the result of our main Theorem \ref{thm1} are stronger, since they give more information on the asymptotic of the metric.
\end{remark}

\subsubsection*{Ordinary \(d\)-tuple points. \( \{P_d =0\} \) consists of \(d\) distinct complex lines \(L_1, \ldots, L_d\)} Let \( 0 < \beta_1 \leq \beta_2 \leq \ldots \leq \beta_d < 1 \). The klt condition  requires that \( 	\sum_{j=1}^{d}(\beta_j -1) > -2 .  \)
We introduce the klt region
\begin{equation*}
\mathcal{K}= \{ (\beta_1, \ldots, \beta_d) \in (0,1)^k \hspace{2mm} \mbox{s.t.} \hspace{2mm} 	-2 < \sum_{j=1}^{d} (\beta_j -1) \}
\end{equation*}
and the Troyanov region
\begin{equation*}
\mathcal{T}= \{ (\beta_1, \ldots, \beta_d) \in (0,1)^k \hspace{2mm} \mbox{s.t.} \hspace{2mm} 	0 < 2 + \sum_{j=1}^{d} (\beta_j -1) < 2 \min_{j} \beta_j \} .
\end{equation*}
These are convex polytopes and \( \mathcal{T} \subset \mathcal{K} \). If \( (\beta_1, \ldots, \beta_d) \in \mathcal{T} \) there is a unique spherical metric \(g_{\beta_1, \ldots, \beta_d}\) on \( \mathbb{CP}^1 \) with cone angles \(2\pi \beta_1, \ldots, \beta_d \) at \(L_1,\ldots, L_d \) and we set \(C(Y)\) to be the corresponding regular PK cone. The area of \(g_{\beta_1, \ldots, \beta_d}\) is \(4 \pi \gamma \) where \(\gamma\) is given by Equation \eqref{number c}. We divide \( \partial \mathcal{T} = \partial_c \mathcal{T} \cup \partial_{nc} \mathcal{T} \) into a collapsing \(\partial_c \mathcal{T} = \{ \sum_{j=1}^{d}(\beta_j -1) = -2 \} \) and a non-collapsing \(\partial_{nc} \mathcal{T} = \{ \sum_{j=1}^{d}(\beta_j -1) > -2 \} \) part. The k.l.t. condition keeps us away from \( \partial_c \mathcal{T} \). On \( \partial_{nc} \mathcal{T} \) we have \( 1- \beta_1 = \sum_{j=2}^{d} (1- \beta_j) \) and \( \lim_{(\beta_1, \ldots, \beta_d) \to \partial_{nc} \mathcal{T} } g_{\beta_1, \ldots, \beta_d} = g_{\beta_1, \beta_1} \). Same as before, when \(2 \beta_1 \geq 2 + \sum_{j=1}^{d}(1-\beta_j) \) we set \(C(Y)= \mathbb{C}_{\beta_1} \times \mathbb{C}_{\gamma^{*}} \) where \(  1- \gamma^{*} = \sum_{j=2}^{d} (1- \beta_j) \). Finally, the volume density function is 
\begin{equation} \label{vol dens ord}
\nu(f, \beta_1, \ldots, \beta_d) = \begin{cases}
\gamma^2 \quad \hspace{2mm} \mbox{if } (\beta_1, \ldots, \beta_d) \in \mathcal{T} \\
\beta_1 \gamma^{*}  \quad  \mbox{if } (\beta_1, \ldots, \beta_d) \in \mathcal{K} \setminus \mathcal{T} \\
\end{cases}
\end{equation}

One can perform computations using valuations similar to the one we did before, and see the emergence of Troyanov's stability conditions.

\subsubsection*{Non-transverse  reducible case, e.g.,  \(f=y(y-x^2)\)} Suppose we put the parameter \(\beta_1\) along \(y=0\) and \(\beta_2\) along \(y=x^2\). Then we can apply the covering trick as before using the map \((u,v)=(x^2,y)\) and looking for PK cone metrics on \(\mathbb{C}^2\) with cone angles equal \(1/2\) along \(u=0\), \(\beta_1\) along \(v=0\), and  \(\beta_2\) along \(u=v\). Using the Troyanov conditions we see that we  have a PK cone metric if \((\beta_1,\beta_2)\) are in the squared rhombus 
\(\beta_2> \pm \beta_1 \mp \frac{1}{2}\) and \(\beta_2< \pm \beta_1 +1 \mp \frac{1}{2}\). The value of the metric density is  equal to \(\nu=\frac{1}{8}(2\beta_1+2\beta_2-1)^2\), and it vanishes precisely on the collapsing boundary corresponding to the klt condition, as usual. There are three different possibilities in which a sequence of parameters \((\beta_1, \beta_2)\) lying in the interior of the rhombus can hit its boundary, corresponding to the jumping of the tangent cones: if we leave from the wall \(\beta_2=-\beta_1+\frac{3}{2}\), by using similar argument as before, we see that tangent cone should be given by \(\mathbb{C}\times \mathbb{C}_{\beta_2+\beta_1-1} \), corresponding to the geometric naive picture of cone angles merging. On the other hand, leaving from the boundary \(\beta_2=\beta_1+\frac{1}{2}\), we see that the tangent cone should be given by the product of two flat cones \( \mathbb{C}_{2\beta_2-1} \times \mathbb{C}_{\beta_1} \). By symmetry, leaving at \(\beta_2=\beta_1-\frac{1}{2}\) gives as cones \( \mathbb{C}_{2\beta_1-1} \times \mathbb{C}_{\beta_2} \). \\

Finally, one should consider the general case where we have a bunch of singular non-reduced non-transverse local curves at $p$. We expect that by performing similar analysis as we did in the previous paragraph, one could give a nice description of the expected behaviour of the tangent cones. However, we haven't investigated in details this situation further. \\

As we mention in the introduction, the actual proof of our main Theorem \ref{thm1} doesn't make deep use of the fact that the ambient space is the projective space or that the singularities are multiple points of lines. Thus, in more generality, we can state the following

\begin{theorem}\label{gen case}
	Let \((X,D)\) be a log klt Calabi-Yau surface with smooth ambient space \(X\). Assume that locally analytically the singularities are modeled on PK cones. Then, for each K\"ahler class, there exist a unique Ricci-flat metric asymptotic to the PK model cones.

	In particular, for example,  \(D\) could have \(A_k\)-singularities \(w^2=z^{k+1} \) in the `stable range': $$\frac{k-1}{2k+2} <\beta < \frac{k+3}{2k+2}.$$
	
\end{theorem} 

\begin{proof} 
	Write \(D = \sum_{j=1}^{n}(1-\beta_j)D_j \). Let \(dV_0\) be a smooth volume form on \(X\) and \(h_j\) be smooth Hermitian metrics on \([D_j]\) together with defining sections \(s_j \in H^0([D_j]) \) of \(D_j\). We consider the measure \(\mu_0 = |s_1|^{2\beta_1-2} \ldots |s_n|^{2\beta_n-2} dV_0 \). The Calabi-Yau condition implies that \(\dd \log \mu_0 \) is trivial in de Rham co-homology. The \(\dd\)-lemma guarantees the existence of a smooth function \(f\), unique up to addition of a constant, such the measure \( \mu_{CY} = e^{-f}\mu_0 \) satisfies \(\dd \log \mu_{CY} \equiv 0 \). Since a pluri-harmonic function is locally the real part of an holomorphic function, it is easy to check that in the complement of \(D\) we can write \(\mu_{CY} = \Omega \wedge \overline{\Omega} \) for a locally defined holomorphic volume form \(\Omega\). 
	
	Thus we are looking to solve the equation $$ \omega_{CY}^2=\mu_{CY}.$$
	
	The construction of the background metric (see section \ref{reference metrics}), the linear analysis and the continuity path (see section \ref{yau cont path sect}) generalizes with essentially no changes. 
	
\end{proof}

For example, the above theorem constructs a Ricci-flat metric on $(\mathbb{C P}^2, (1-\beta_1) C_1+ (1-\beta_2) C_2)$  where $C_2$ is a generic line, $C_1$, say, the projectivization  of the cusp $w^2=z^3$, and the cone angles satisfy the conditions $3(1-\beta_1)+(1-\beta_2)=3$, $1/6< \beta_1 < 5/6$ and $0<\beta_2<1$.

\subsection*{Chern-Weil formula and constant holomorphic sectional curvature}

The discussion about Equation \eqref{energy formula} suggests a natural generalization of the energy formula to the more general setting of a weak KEcs on a two dimensional klt pair \((X, \sum_{j=1}^{n} (1-\beta_j) D_j)\). We propose the following.

\begin{conjecture}\label{conj:E2}
	We expect that
	\begin{equation} \label{CHERNWEILintro}
	3 c_2 (X, D) - c_1 (X, D)^2 = \frac{3}{8 \pi^2} \int_{X} |\mathring{ \mbox{Rm} }|^2 dV_g
	\end{equation}
	where \( \mathring{ \mbox{Rm} } \) is the trace free Riemann curvature tensor and we have introduced the `logarithmic' second Chern class 
	\begin{equation} \label{log c2} 
	c_2(X, D) = \chi (X) + \sum_{j=1}^{n} (\beta_j -1) \chi (D_j \setminus S ) + \sum_{p \in S} (\nu(p) -1) 
	\end{equation}
	with \( S = \mbox{Sing}(X, D) \)  the finite set of points where the curve \(C = \cup_{j=1}^n D_j \) is singular and 
	we have identified \( c_1(X, D)^2 \in H^{2, 2}(X, \mathbb{R}) \cong \mathbb{R} \) via integration over \(X\).
\end{conjecture}

Conjecture \ref{conj:E2} is considerable more difficult to prove than Conjecture \ref{conj:E} 
because the extra difficulty caused by the tangent cone jump at a point \(p \in S\), however recent advances following \cite{szekelyhidi} (in particular \cite{dBE2}) shed some light in this direction. 

An immediate consequence of Equation \ref{CHERNWEILintro} is a logarithmic version of the Bogomolov-Miyaoka-Yau inequality which matches with theorem 0.1 in \cite{Langer}. More precisely Langer asserts that 
\begin{equation} \label{langer}
3 e_{orb}(X, D) \geq (K_X + D)^2
\end{equation}
where  the orbifold Euler number is defined as
\begin{equation*}
e_{orb} (X, D) = \chi (X) + \sum_{j=1}^{n} (\beta_j -1) \chi (D_j \setminus S ) + \sum_{p \in S} (e_{orb}(p; X, D) -1) 
\end{equation*}
and the local orbifold Euler number \(e_{orb}(p; X, D)\) is an analytic invariant of the singularity defined in terms of Chern numbers of sheaves of rational \(1\)-forms with logarithmic poles. Equation \eqref{langer} is an immediate consequence of Equation \eqref{CHERNWEILintro} provided that the local orbifold Euler number agrees with the volume density. We have checked that \( e_{orb} (p; X, D) = \nu(p) \) in all cases contemplated in sections 8 and 9 in \cite{Langer}.

\begin{remark}
	Soon after our preprint Chi Li \cite{ChiLi} proved that the two quantities agree for klt, and even more general log canonical, pairs invariant under a \(\mathbb{C}^*\) action.
\end{remark}

Our approach has the advantage of interpreting the equality case in Equation \eqref{langer} as a characterization of the conically singular constant holomorphic sectional curvature metrics and provides a unified framework for previous treatments on these interesting geometries (which to the authors' knowledge have only been considered in \cite{panov} and \cite{CouwenbergHeckmanLooijenga}). Given a pair we can regard the quantity \( 3c_2(X, D) - c_1(X, D)^2 = F (\mu_1, \ldots, \mu_n) \) as a positive semi-definite quadratic function of the weights \(\mu_j = 1-\beta_j\). If \(X = \mathbb{CP}^2 \) then \(F(0) =0 \) and if the divisor is made up of lines then the combinatorics of the arrangement determines \(F\). The arrangements for which the main diagonal \((\mu, \ldots, \mu)\) is in the kernel of \(F\) were characterized by Hirzebruch by \(n=3k\) and each
line intersects the other at exactly \(k+1\) points; when  \( \beta = \beta_j = (k-1)/k\) for all \(j\) (CY case) Panov showed the existence of the corresponding PK metric. We further expect that for any of these arrangements the diagonal give us a one-parameter family of constant holomorphic sectional curvature metrics, positive if \((k-1)/k < \beta < 1 \) and negative if \( \beta < (k-1)/k \). If \(d\) denotes the maximum number of lines intersecting at one point then the klt condition is equivalent to \( d(1-\beta) < 2 \); as \(\beta\) tends to the lower bound \( (d-2)/d \) we expect the metrics to converge to conical complex hyperbolic metrics with cuspidal points previously discovered by Hirzebruch.

\subsection*{Chern-Weil formulas and weighted  invariants  in higher dimensions}

We expect in higher dimension that the BMY inequality in the logarithmic setting (as the one considered in \cite{GT}) can also be improved by modifying the second Chern class by weighting the codimension one singularities of the divisor of the pair with the volume density/infiumum of normalized volume of valuations (which should be the same at the generic point, see \cite{Liu}).

More generally, it would be interesting to consider `weighted Chern classes', using the volume density/infiumum of normalized volume of valuations, which should compute Chern-Weil integrals for nice K\"ahler metrics on klt pairs. This is interesting also in the absolute (i.e., non-conical) case.

For example, one could consider a \emph{weighted Euler Characteristic}  of the weighted klt pair $(X,D)$, with respect to the normalized volume  $\nu: X \rightarrow (0,1]$:
\begin{equation*}
e_{\nu} (X, D) = \sum_t t e(\nu^{-1}(t)) = e(X^0\setminus{D})+ \beta e(D^0)+ \dots \in \overline{\mathbb{Q}}(\beta)
\end{equation*}
Here $e$ denotes the Euler characteristic with compact support  and $\overline{\mathbb{Q}}$ the algebraic closure of the rational number.  The above formula makes sense provided that the following two natural properties hold: the infiumum of normalized volume of valuations $\nu$ is a constructible function on $X$ and $\nu$ takes only finitely many values within $(0,1]$. Moreover, since we expect the normalized volume satisfies $\nu_{X\times Y}=\nu_X\nu_Y$, we would have that the weighted Euler characteristic is multiplicative under products (similarly for finite coverings), thus giving a potentially interesting new motivic invariant of singular klt varieties.  In two dimensions, note that $e_{\nu} (X, D)=c_2(X,D)$ of formula \eqref{log c2}. 

Moreover, such invariant $e_{\nu} (X, D)$ should agree with the Chern-Weil integral of the form representing the top Chern class of a K\"ahler metric with the right tangent cones. For some results which share some similarity with the picture above described see \cite{McM}.
\appendix

\section{Polyhedral metrics on \(\mathbb{CP}^1\) and their moduli}  \label{pm sphere section}

Let \(n \geq 3\). Fix \(\beta = (\beta_1, \ldots, \beta_n) \) with \( 0 < \beta_j < 1 \) and 
\begin{equation} \label{flat condition}
\sum_{j=1}^{n}(1-\beta_j) =2 .
\end{equation}
The basic fact (see \cite{troyanovSE}) is that there is a correspondence between the moduli space of configurations of \(n\) distinct ordered points in the projective line \(\mathcal{M}_{0, n}\) and flat metrics of a fixed volume on the two sphere with cone angles \(2\pi\beta_1, \ldots, 2\pi\beta_n\). To make this correspondence a bijection one needs to quotient \(\mathcal{M}_{0, n}\) by the action of elements of the symmetric group which permute points with equal cone angle, but we shall ignore this fact (or either restrict to the case where all cone angles are distinct). The point is that \(\mathcal{M}_{0, n}\) has a natural structure of a non-compact complex affine variety (of complex dimension \(n-3\)) and the above correspondence provides a natural Gromov-Hausdorff completion \(\overline{\mathcal{M}}_{0, n}^{\beta, GH}\) which we would like to understand from an algebraic perspective.

The correspondence between configurations of points and polyhedral metrics is quite explicit: Given \(x_1, \ldots, x_n \in \mathbb{CP}^1 \) we take an affine coordinate \(\xi\) such that \(\xi (a_j) = x_j\) with \(a_j \in \mathbb{C}\) and define
\begin{equation} \label{polyhedral metric eq}
g = |\xi-a_1|^{2\beta_1-2} \ldots |\xi-a_n|^{2\beta_n-2} |d\xi|^2 .
\end{equation}
It is easy to check that \(g\) is a flat metric on \(\mathbb{CP}^1\) with cone angle \(2\pi\beta_j\) at \(x_j\) and it is also easy to prove that it is unique up to scaling by a constant factor. Equation \eqref{flat condition} guarantees that \(g\) extends smoothly over \(\xi=\infty\). We make the aside remark that if \( 0 < \sum_{j=1}^{n}(1-\beta_j) < 1 \) then the metric \(g\) given by Equation \eqref{polyhedral metric eq} defines a flat metric on \(\mathbb{C}\) asymptotic to the cone \(\mathbb{C}_{\gamma}\) with \( 1 - \gamma = \sum_{j=1}^{n}(1-\beta_j) \). If \(\sum_{j=1}^{j} (1-\beta_j) =1 \) then \(g\) is asymptotic to the cylinder \(|\xi|^{-2} |d\xi|^2\). If \( \sum_{j=1}^{n}(1-\beta_j) >1 \) then \(g\) defines a flat metric on \(\mathbb{CP}^1\) with cone angle \(2\pi\beta_{\infty}\) at \(\xi=\infty\) with \( 1 + \beta_{\infty} =  \sum_{j=1}^{n}(1-\beta_j) \).
We can write the volume form of \(g\) as \(\omega = i \Omega \wedge \overline{\Omega}\) where \(\Omega\) is the locally defined holomorphic \(1\)-form 
\begin{equation} \label{pol met vol}
\Omega = (\xi-a_1)^{\beta_1 -1} \ldots (\xi- a_n)^{\beta_n -1} d\xi .
\end{equation}

Any metric given by Equation \eqref{polyhedral metric eq} is isometric to either the boundary of a convex polytope in \(\mathbb{R}^3\) or the double of a polygon, see \cite{thurstonSP, troyanovSE} and references therein. The second case arises precisely when we can move the conical points with a M\"obius map so that \(a_1, \ldots, a_n \in \mathbb{R}\); then \(g\) is invariant under \(\xi \to \overline{\xi}\) and its restriction to the upper half space is equal to the pull-back of the euclidean metric by the Riemann mapping \(F: \mathbb{H} \to R \) where \(R\) is the polygon whose double recovers \(g\) and \(F\) is the Schwarz-Christoffel integral
\begin{equation*}
F(\xi) = \int_{0}^{\xi} (\eta-a_1)^{\beta_1 -1} \ldots (\eta- a_n)^{\beta_n -1} d\eta 
\end{equation*}
which we can also think as the enveloping map for \(g\).

It was shown by Thurston \cite{thurstonSP} that \(\overline{\mathcal{M}}_{0, n}^{\beta, GH}\) has a natural structure of a complex hyperbolic cone manifold. There is a  smooth complex hyperbolic metric \(g_{\beta}\) on \(\mathcal{M}_{0, n}\) specified by coordinates to \(\mathbb{CH}^{n-3}\) given by the periods \(\int_{a_i}^{a_j} \Omega \). The metric \(g_{\beta}\) has conical singularities along a boundary divisor which represents collisions between at least two conical points, this divisor has a further stratification labelled by partitions \(\mathcal{P}\) of \( \{1, \ldots, n\} \) which record the set of points colliding. The conical angle of \(g_{\beta}\) at points of the boundary divisor is given explicitly in terms of \(\beta\) and \(\mathcal{P}\). If there is no \(B \subset \{1, \ldots, n\} \) such that \(\sum_{j \in B} (1-\beta_j) =1 \) then \(\overline{\mathcal{M}}_{0, n}^{\beta, GH}\) is compact and it is natural to relate it (for rational values of the angles) to the GIT quotient of \(\mathbb{P}^1 \times \ldots \mathbb{P}^1 \) with polarization \(\mathcal{O}(1-\beta_1) \times \ldots \times \mathcal{O}(1-\beta_n) \).

In higher dimensions we consider Calabi-Yau (and more generally K\"ahler-Einstein) metrics with conical singularities along hyperplane arrangements. Under  appropriate numerical conditions on the cone angles there is a correspondence between line arrangements and KEcs. In the Fano case Fujita \cite{fujita} has shown that log K-stability of a hyperplane arrangement is equivalent to GIT stability of the corresponding configuration in \((\mathbb{P}^n)^{*} \times \ldots \times (\mathbb{P}^n)^{*} \) with polarization \(\mathcal{O}(1-\beta_1) \times \ldots \times \mathcal{O}(1-\beta_n) \) and this indicates that the Gromov-Hausdorff compactification of the corresponding KEcs agrees with this GIT quotient.

\bibliographystyle{plain}
\bibliography{LAref}

\end{document}